\documentclass[12pt,a4paper]{article}%
\usepackage[utf8]{inputenc}
\usepackage{amsmath}
\usepackage{amsfonts}
\usepackage{amssymb}
\usepackage{graphicx}%
\providecommand{\U}[1]{\protect\rule{.1in}{.1in}}
\newtheorem{theorem}{Theorem}

\newtheorem{corollary}[theorem]{Corollary}

\newtheorem{definition}[theorem]{Definition}

\newtheorem{lemma}[theorem]{Lemma}

\newtheorem{problem}[theorem]{Problem}

\newtheorem{remark}[theorem]{Observation}

\newenvironment{proof}[1][Proof]{\noindent\textbf{#1.} }{\ \hfill \rule{0.5em}{0.5em}\bigskip}
\graphicspath{{D:/Dropbox/Riste-Sedlar/MetricDim/Paper4/Slike/}}

\begin{document}

\title{Vertex and edge metric dimensions of unicyclic graphs}
\author{Jelena Sedlar$^{1}$,\\Riste \v Skrekovski$^{2,3}$ \\[0.3cm] {\small $^{1}$ \textit{University of Split, Faculty of civil
engineering, architecture and geodesy, Croatia}}\\[0.1cm] {\small $^{2}$ \textit{University of Ljubljana, FMF, 1000 Ljubljana,
Slovenia }}\\[0.1cm] {\small $^{3}$ \textit{Faculty of Information Studies, 8000 Novo
Mesto, Slovenia }}\\[0.1cm] }
\maketitle

\begin{abstract}
In a graph $G$, the cardinality of the smallest ordered set of vertices that
distinguishes every element of $V(G)$ (resp. $E(G)$) is called the vertex
(resp. edge) metric dimension of $G$. In \cite{SedSkreBounds} it was shown
that both vertex and edge metric dimension of a unicyclic graph $G$ always
take values from just two explicitly given consecutive integers that are
derived from the structure of the graph. A natural problem that arises is to
determine under what conditions these dimensions take each of the two possible
values. In this paper for each of these two metric dimensions we characterize
three graph configurations and prove that it takes the greater of the two
possible values if and only if the graph contains at least one of these
configurations. One of these configurations is the same for both dimensions,
while the other two are specific for each of them. This enables us to
establish the exact value of the metric dimensions for a unicyclic graph and
also to characterize when each of these two dimensions is greater than the
other one.

\end{abstract}

\textit{Keywords:} vertex metric dimension; edge metric dimension; unicyclic
graphs; cactus graphs.

\textit{AMS Subject Classification numbers:} 05C12; 05C76

\section{Introduction}

In this paper we consider only simple and connected graphs unless explicitly
stated otherwise. The distance between a pair of vertices $u$ and $v$ in a
graph $G$ is denoted by $d(u,v)$. We say that a vertex $s$ from $G$
\emph{distinguishes} or \emph{resolves} a pair of vertices $u$ and $v$ from
$G$ if $d(s,u)\not =d(s,v).$ A set of vertices $S\subseteq V(G)$ is called a
\emph{vertex metric generator,} if every pair of vertices in $G$ is
distinguished by at least one vertex from $S.$ The cardinality of a smallest
vertex generator in $G$ is called the \emph{vertex metric dimension} of $G,$
and it is denoted by $\mathrm{dim}(G).$ For this variant of metric dimension
the prefix "vertex" is sometimes omitted, so we say only metric generator and
metric dimension.

The concept of metric generator was introduced in \cite{SlaterVertex} under
the name of locating set and the problem of uniquely identifying the location
of an intruder in a network by its distances to the locating devices was
considered. Independently, the same concept was studied in \cite{HararyVertex}
under the name of resolving set. The complexity of approximating the metric
dimension of a graph was studied in \cite{KhullerVertex}, applications of
metric dimension in digital geometry in \cite{MelterVertex}, the comparison of
metric dimension of graphs and line graphs in \cite{KleinVertex}, how the
metric dimension can be affected by the addition of a single vertex in
\cite{BuczkowskiVertex}, vertices contained in all metric generators in
\cite{HakanenYero}, and the behaviour of metric dimension with respect to
various graph operations in \cite{ChartrandVertex, SaputroVertex}.

Recently, the concept of metric dimension was extended from distinguishing
vertices to distinguishing edges. Similarly as above, a vertex $s\in V(G)$
\emph{distinguishes} two edges $e,f\in E(G)$ if $d(s,e)\neq d(s,f)$, where
$d(e,s)=d(uv,s)=\min\{d(u,s),d(v,s)\}$. The authors of \cite{TratnikEdge}
noticed that there are graphs in which none of the smallest metric generators
distinguishes all pairs of edges, so they were motivated to introduce a notion
of an \emph{edge metric generator} as any set $S\subseteq V(G)$ which
distinguishes all pairs of edges, the \emph{edge metric dimension} (denoted by
$\mathrm{edim}(G)$) as the cardinality of a smallest edge metric generator,
and then study its relation with the vertex metric dimension. They presented
families of graphs for which $\mathrm{dim}(G)<\mathrm{edim}(G)$, or
$\mathrm{dim}(G)=\mathrm{edim}(G)$, or $\mathrm{dim}(G)>\mathrm{edim}(G)$,
also they established that determining the edge metric dimension of a graph is NP-hard.

This new variant of metric dimension immediately attracted a lot of interest.
The behaviour of edge metric dimension on several graph operations was studied
in \cite{ZubrilinaEdge}, the edge metric dimension of convex polytopes and its
related graphs in \cite{ZhangGaoEdge}, graphs with the maximum edge metric
dimension in \cite{ZhuEdge}, pattern avoidance in graphs with bounded metric
dimension or edge metric dimension in \cite{GenesonEdge}, an approximation
algorithm for the edge metric dimension problem in
\cite{HuangApproximationEdge}, comparison of metric dimensions in
\cite{KelHypercubes, Knor}, bounds on vertex and edge metric dimension of
graphs with edge disjoint cycles in \cite{SedSkreBounds}.

Recently the mixed metric dimension was also introduced \cite{KelencMixed},
where a mixed metric generator of a graph $G$ is defined as any set $S$ which
distinguishes all pairs from $V(G)\cup E(G),$ the size of a smallest such set
is the \emph{mixed metric dimension} of $G$, and it is denoted by
$\mathrm{mdim}(G)$. The paper \cite{KelencMixed} contains lower and upper
bounds for various graph classes. The mixed metric dimension was further
studied in \cite{SedSkrekMixed} for graphs with edge disjoint cycles, these
results were further generalized to graphs with prescribed cyclomatic number
in \cite{SedSkreTheta}. For a wider and systematic introduction of the topic
of these three variants of metric dimension, we recommend the PhD thesis of
Kelenc \cite{KelPhD}.

The focus of this paper is on the result from \cite{SedSkreBounds}, where it
was was established that both $\mathrm{dim}(G)$ and $\mathrm{edim}(G)$ of a
unicyclic graph $G$ can take its value from only two consecutive integers
which can be determined from the structure of the graph. In this paper, we go
further and characterize when these two metric dimensions take each of the two
possible values, a research direction which is proposed in \cite{Knor}. This
promptly resolves which of the following three situations $\mathrm{dim}%
(G)<\mathrm{edim}(G)$, or $\mathrm{dim}(G)=\mathrm{edim}(G)$, or
$\mathrm{dim}(G)>\mathrm{edim}(G)$ holds for a unicyclic graph $G$.

\section{Preliminaries}

Throughout the paper we will use the following notation. The cycle in a
unicyclic graph $G$ is denoted by $C=v_{0}v_{1}\cdots v_{g-1}$, where $g$ is
the length of $C$ (i.e. $g=\left\vert V(C)\right\vert $). Additionally, each
edge $v_{i}v_{i+1}$ of $C$ is denoted by $e_{i}.$ The connected component of
$G-E(C)$ containing vertex $v_{i}$ is denoted by $T_{v_{i}}$. A \emph{thread}
in a graph $G$ is a path $u_{1}u_{2}\cdots u_{k}$ in which $u_{k}$ is of
degree $1$, all other vertices of the thread are of degree $2$ and the vertex
$u_{1}$ is a neighbour of a vertex $v\in V(G)$ with $\mathrm{\deg}(v)\geq3$.
For a vertex $v$ from a unicyclic graph $G$ we say that it is a
\emph{branching} vertex if $v\not \in V(C)$ and $\deg(v)\geq3$ or $v\in V(C)$
and $\deg(v)\geq4$. We say that a vertex $v_{i}\in V(C)$ is
\emph{branch-active} if $T_{v_{i}}$ contains a branching vertex. Let us denote
by $b(C)$ the number of all branch-active vertices on $C$. As the cycle $C$ is
the only cycle in a unicyclic graph, we may use notation $b(G)$ instead of
$b(C)$ as well.

Notice that every branching vertex $v$ which belongs to $T_{v_{i}}$ has at
least two neighbours which are not distinguished by any vertex from outside of
$T_{v_{i}},$ even more - it may not be distinguished by some vertices from
$T_{v_{i}}$, see Figure \ref{FigureBranching}.a) for illustration. Similarly
holds for a pair of edges incident to branching vertices $v$ and $v_{k.}$ We
say that a set $S\subseteq V(G)$ of a (unicyclic) graph $G$ is
\emph{branch-resolving} if for every $v\in V(G)$ of degree at least $3$, the
set $S$ contains a vertex from all threads hanging at $v$ except possibly from
one such thread, see Figure \ref{FigureBranching}.b).

\begin{figure}[h]
\begin{center}
$%
\begin{array}
[c]{cccc}%
\text{a)} &
\text{\raisebox{-1\height}{\includegraphics[scale=0.5]{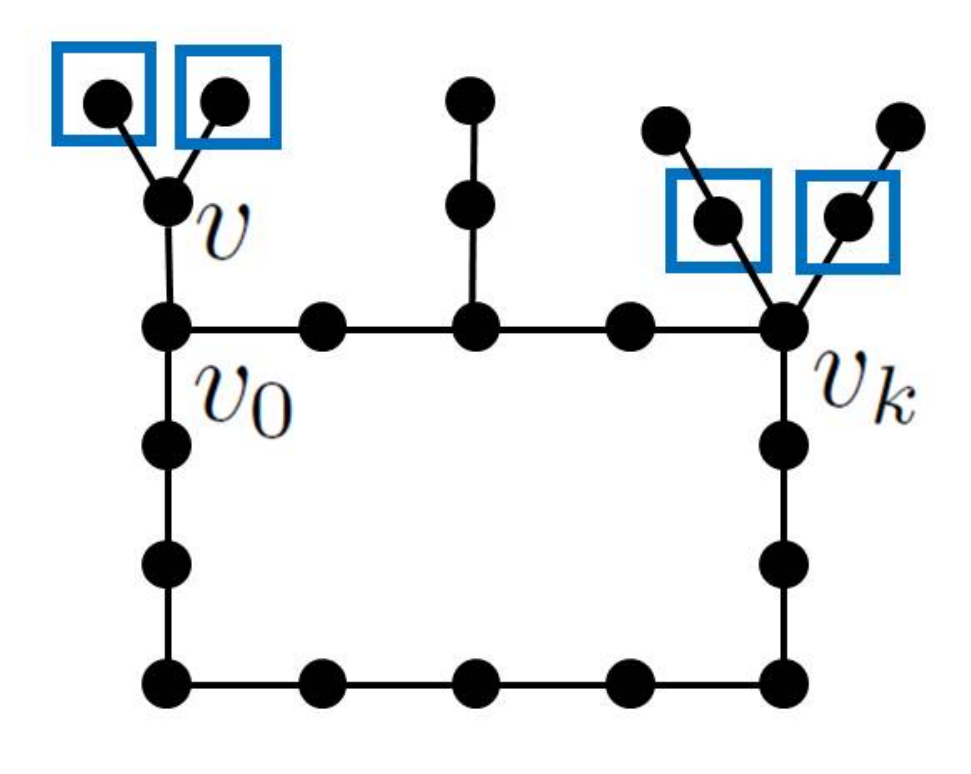}}} &
\text{b)} &
\text{\raisebox{-1\height}{\includegraphics[scale=0.5]{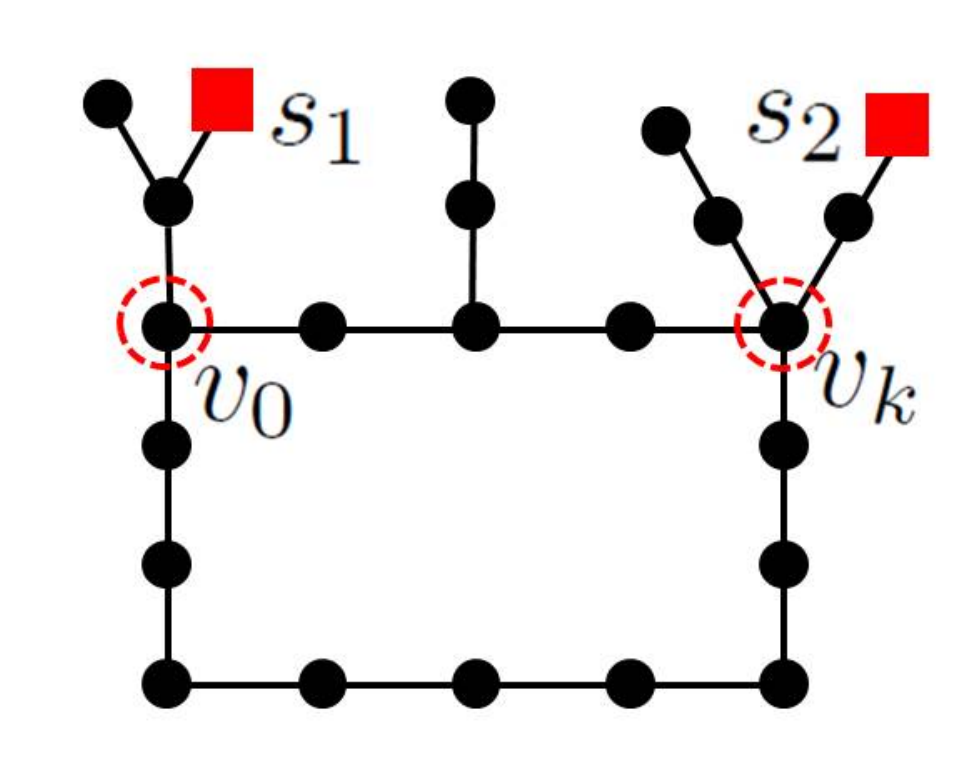}}}%
\end{array}
$
\end{center}
\caption{Both figures show the same graph for which: a) Two branching vertices
are $v$ and $v_{k},$ and each of them has a pair of neighbours which cannot be
distinguished by a vertex outside of $T_{v_{0}}$ and $T_{v_{k}}$ respectively.
b) $S=\{s_{1},s_{2}\}$ is an example of a smallest branch-resolving and the
set $\{v_{0},v_{k}\}$ is both the set of $S$-active vertices, and the set of
branch-active vertices.}%
\label{FigureBranching}%
\end{figure}

Let us denote by $\ell(v)$ the number of all threads attached to a vertex $v$
of degree $\geq3$ and let
\[
L(G)=\sum(\ell(v)-1),
\]
where the sum runs over all vertices $v$ of degree $\geq3$ of $G$ with
$l(v)>1.$ Note that for every branch-resolving set $S$ we have $\left\vert
S\right\vert \geq L(G)$ where equality holds if $S$ is a branch-resolving sets
of minimum cardinality.

For a set of vertices $S\subseteq V(G)$ we say that $v_{i}$ from $C$ is
$S$\emph{-active} if $T_{v_{i}}$ contains a vertex from $S$. The number of
$S$-active vertices on $C$ is denoted by $a_{S}(C)$. Note that for a
branch-resolving set $S$ it holds that $a_{S}(C)\geq b(C)$ with equality
holding for a smallest branch-resolving set $S$ in $G$. A smallest
branch-resolving set usually is not unique, but all smallest branch-resolving
sets have the same set of $S$-active vertices on the cycle which equals the
set of branch-active vertices on the cycle (see again Figure
\ref{FigureBranching}.b)). Finally, we say that a thread hanging at a vertex
$v$ of degree $\geq3$ is $S$\emph{-free} if it does not contain a vertex from
$S.$ Let us remark, as $v$ is not a vertex of the thread, if $v\in S,$ it does
not prevent the thread to be $S$-free.

The following two properties of branch-resolving sets were shown in
\cite{SedSkreBounds}.

\begin{lemma}
\label{Lemma_a(S)vj2} Let $S$ be a metric generator or an edge metric
generator of a unicyclic graph $G$. Then $S$ is a branch-resolving set with
$a_{S}(C)\geq2$.
\end{lemma}

\begin{lemma}
\label{Lemma_SameComponent} Let $G$ be a unicyclic graph and $S\subseteq V(G)$
a branch-resolving set with $a_{S}(C)\geq2$. Then, any two vertices (also any
two edges) from a same connected component of $G-E(C)$ are distinguished by
$S$.
\end{lemma}

We say that a set $S\subseteq V(G)$ is \emph{biactive} if $a_{S}(C)\geq2.$
Thus, according to Lemma \ref{Lemma_a(S)vj2}, if the set $S$ is not
branch-resolving or if it is not biactive, then $S$ certainly is not a vertex
nor an edge metric generator. The problem of non-distinguished pairs of
vertices (resp. edges), if $S$ is not branch-resolving, is already illustrated
by Figure \ref{FigureBranching}.a). Let us now consider when $S$ is not
biactive. If $a_{S}(C)=0$ then $S=\phi$ and $S$ cannot be a metric generator,
on the other hand if $a_{S}(C)=1$ then the pair of vertices (resp. edges) from
$C$ which are adjacent (resp. incident) to the only $S$-active vertex on $C$
are not distinguished by $S$, see Figure \ref{FigureSactive}.a) for illustration.

\begin{figure}[h]
\begin{center}
$%
\begin{array}
[c]{cccc}%
\text{a)} &
\text{\raisebox{-1\height}{\includegraphics[scale=0.5]{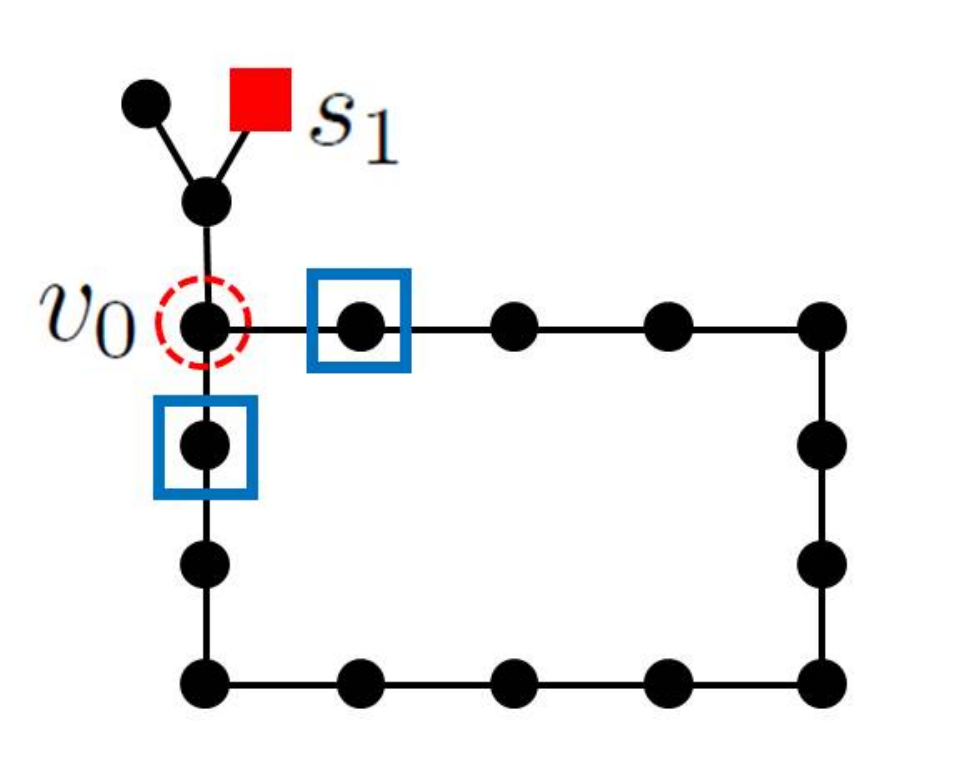}}} &
\text{b)} &
\text{\raisebox{-1\height}{\includegraphics[scale=0.5]{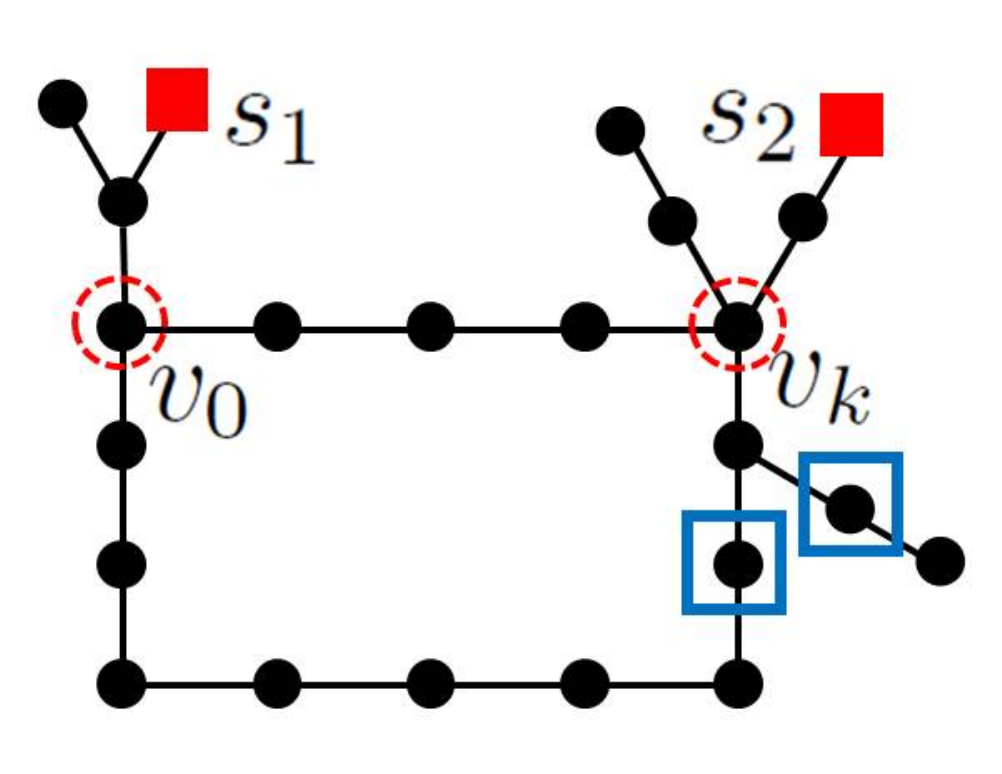}}}%
\end{array}
$
\end{center}
\caption{An example of a branch-resolving set $S$ with: a) $a_{S}(C)=1,$ b)
$a_{S}(C)=2,$ which is not a vertex metric generator. A pair of
non-distinguished vertices is marked in the figure.}%
\label{FigureSactive}%
\end{figure}

Yet, a biactive branch-resolving set $S$ may or may not be a vertex (edge)
metric generator, this depends also of the position of vertices on the cycle
which have an $S$-free thread attached. For example, the set $S$ from Figure
\ref{FigureBranching}.b) is a biactive branch-resolving set which is a vertex
(edge) metric generator, and the set $S$ from Figure \ref{FigureSactive}.b) is
also a biactive branch-resolving set but it is not a vertex (edge) metric
generator. Notice that even if we add to $S$ all vertices of $T_{v_{i}},$ for
$0\leq i\leq k,$ the set $S$ would still not be a vertex (edge) metric generator.

By the above, we need to introduce a configuration of $S$-active vertices on
the cycle $C$ which will suffice for a biactive branch-resolving set $S$ to
become a vertex (edge) metric generator. For this, let $v_{i}$, $v_{j}$, and
$v_{k}$ be three vertices of the cycle $C$. We say that $v_{i}$, $v_{j}$, and
$v_{k}$ form a \emph{geodesic triple} of vertices on $C$, if
\[
d(v_{i},v_{j})+d(v_{j},v_{k})+d(v_{i},v_{k})=|V(C)|.
\]
Observe that for any two vertices of $C$, we can easily choose a third one
such that they form a geodesic triple. It is also easy to observe that a
geodesic triple of vertices distinguishes vertices from $C.$ Moreover the
following property of geodesic triples was shown in \cite{SedSkreBounds}.

\begin{lemma}
\label{Lemma_GeodTrip} Let $G$ be a unicyclic graph, and let $S$ be a
branch-resolving set of $G$ with $a_{S}(C)\geq3$ and with three $S$-active
vertices on $C$ forming a geodesic triple. Then, $S$ is both a metric
generator and an edge metric generator of $G$.
\end{lemma}

\noindent By the all above, for a set of vertices $S\subseteq V(G)$ we can
conclude the following:

\begin{enumerate}
\item[S1.] If the set $S$ is not a biactive branch-resolving set, then $S$
cannot be a vertex (resp. edge) metric generator.

\item[S2.] If the set $S$ is a biactive branch-resolving set, then either:

\begin{enumerate}
\item $S$ is a vertex (resp. edge) metric generator by itself; or

\item $S$ is not a vertex (resp. edge) metric generator by itself, and it
suffices to introduce precisely one more vertex to $S$ in order to become a
vertex (resp. edge) metric generator according to Lemma \ref{Lemma_GeodTrip}.
\end{enumerate}
\end{enumerate}

In this paper we will establish necessary and sufficient condition under which
that one additional vertex must be introduced to a smallest biactive
branch-resolving set $S$ to become a vertex (resp. edge) metric generator. In
order to do so, we will first consider unicyclic graphs with $b(G)\geq2$ and
identify three structural configurations $\mathcal{A}$, $\mathcal{B}$, and
$\mathcal{C}$ (resp. $\mathcal{A}$, $\mathcal{D}$, and $\mathcal{E}$) in such
a graph which are the only obstacle for a smallest biactive branch-resolving
set to be a vertex (resp. edge) metric generator. Since these configurations
in graphs with $b(G)<2$ depend on the set $S,$ this approach is further
extended by introducing a more general property of a graph which enables us to
derive results which encapsulate also graphs with $b(G)<2,$ namely
$\mathcal{ABC}$-positivity and $\mathcal{ABC}$-negativity (resp.
$\mathcal{ADE}$-positivity and $\mathcal{ADE}$-negativity).

For characterization of the smallest biactive branch-resolving sets $S$ that
need to be introduced an additional vertex in order to become a metric
generator, the position of $S$-active vertices on the cycle $C$ matters. In
order to be able to deal with them, we introduce the following labelling of
the cycle.

\begin{definition}
Let $G$ be a unicyclic graph with the cycle $C$ of length $g$ and let $S$ be a
biactive branch-resolving set in $G$. We say that $C=v_{0}v_{1}\cdots v_{g-1}$
is \emph{canonically }labelled with respect to $S$ if $v_{0}$ is $S$-active
and $k=\max\{i:v_{i}$ is $S$-active$\}$ is as small as possible.
\end{definition}

\noindent Notice that when there is no geodesic triple of $S$-active vertices,
the canonical labelling implies $k\leq g/2$. In particular, if $a_{S}(C)=2$,
then $k\leq g/2$. Throughout the paper we will assume that the cycle $C$ is
canonically labelled with respect to the given biactive branch-resolving set
$S$, unless explicitly stated otherwise.

\section{Vertex metric dimension}

Regarding S2 we want to characterize when a smallest biactive branch-resolving
set is a vertex metric generator by itself, and when an additional vertex must
be added to such a set to become a vertex metric generator. For this, we
introduce three structural configurations $\mathcal{A}$, $\mathcal{B}$, and
$\mathcal{C}$ of the graph $G$ with respect to $S$ which will be crucial for
the characterization.

\begin{definition}
Let $G$ be a unicyclic graph, and let $S$ be a biactive branch-resolving set
in $G$. We say that the graph $G$ with respect to $S$ \emph{contains} configurations:

\begin{description}
\item {$\mathcal{A}$}. If $a_{S}(C)=2$, $g$ is even, and $k=g/2$;

\item {$\mathcal{B}$}. If $k\leq\left\lfloor g/2\right\rfloor -1$ and there is
an $S$-free thread hanging at a vertex $v_{i}$ for some $i\in\lbrack
k,\left\lfloor g/2\right\rfloor -1]\cup\lbrack\left\lceil g/2\right\rceil
+k+1,g-1]\cup\{0\}$;

\item {$\mathcal{C}$}. If $a_{S}(C)=2$, $g$ is even, $k\leq g/2$ and there is
an $S$-free thread of the length $\geq g/2-k$ hanging at a vertex $v_{i}$ for
some $i\in\lbrack0,k]$.
\end{description}
\end{definition}

Notice that configuration $\mathcal{C}$ with $k=g/2$ is also configuration
$\mathcal{A}$, and configuration $\mathcal{C}$ with $i\in\{0,k\}$ and $k\leq
g/2-1\ $is also configuration $\mathcal{B}$.

\begin{figure}[h]
\begin{center}
$%
\begin{array}
[c]{cccc}%
\text{a)} &
\text{\raisebox{-1\height}{\includegraphics[scale=0.5]{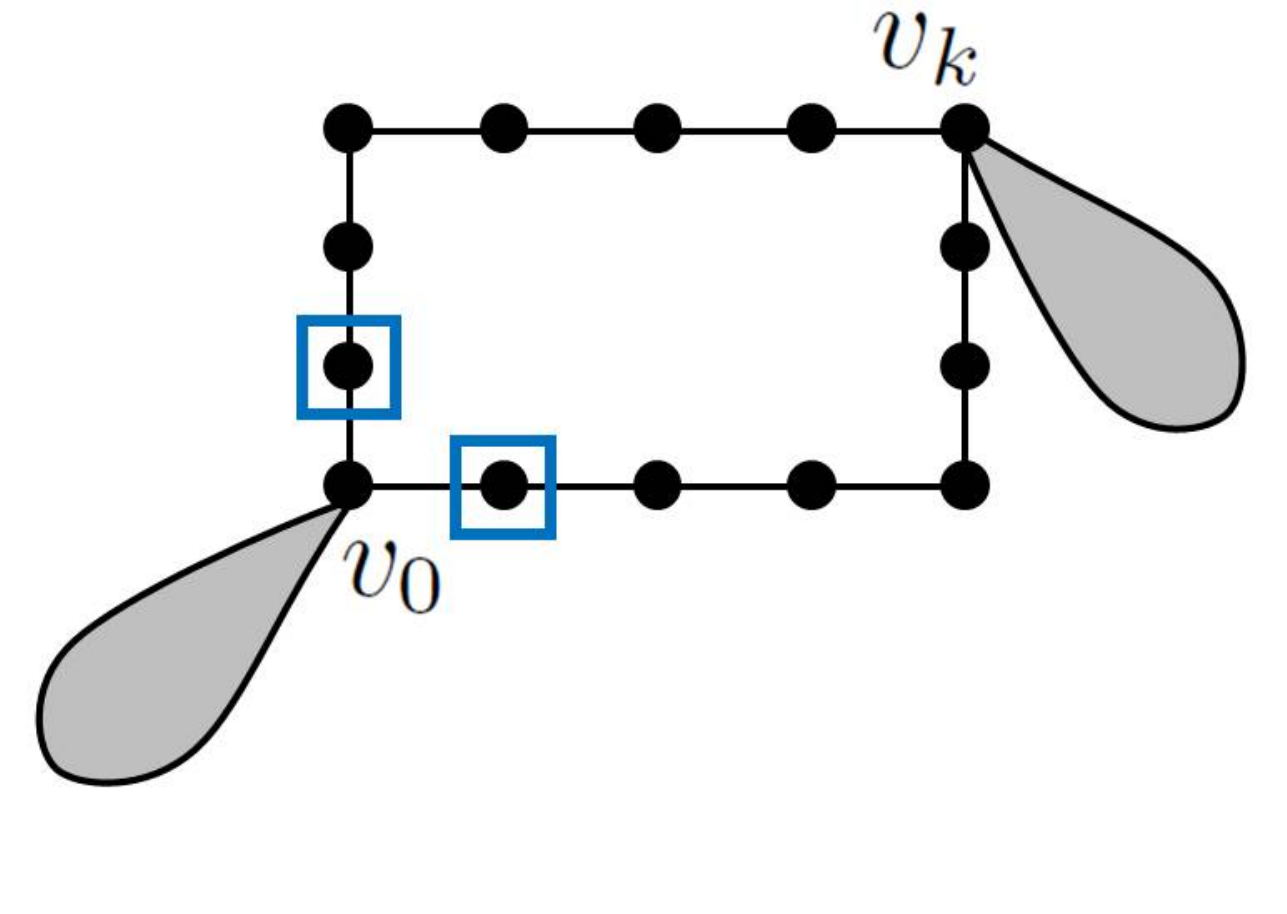}}} &
\text{b)} &
\text{\raisebox{-1\height}{\includegraphics[scale=0.5]{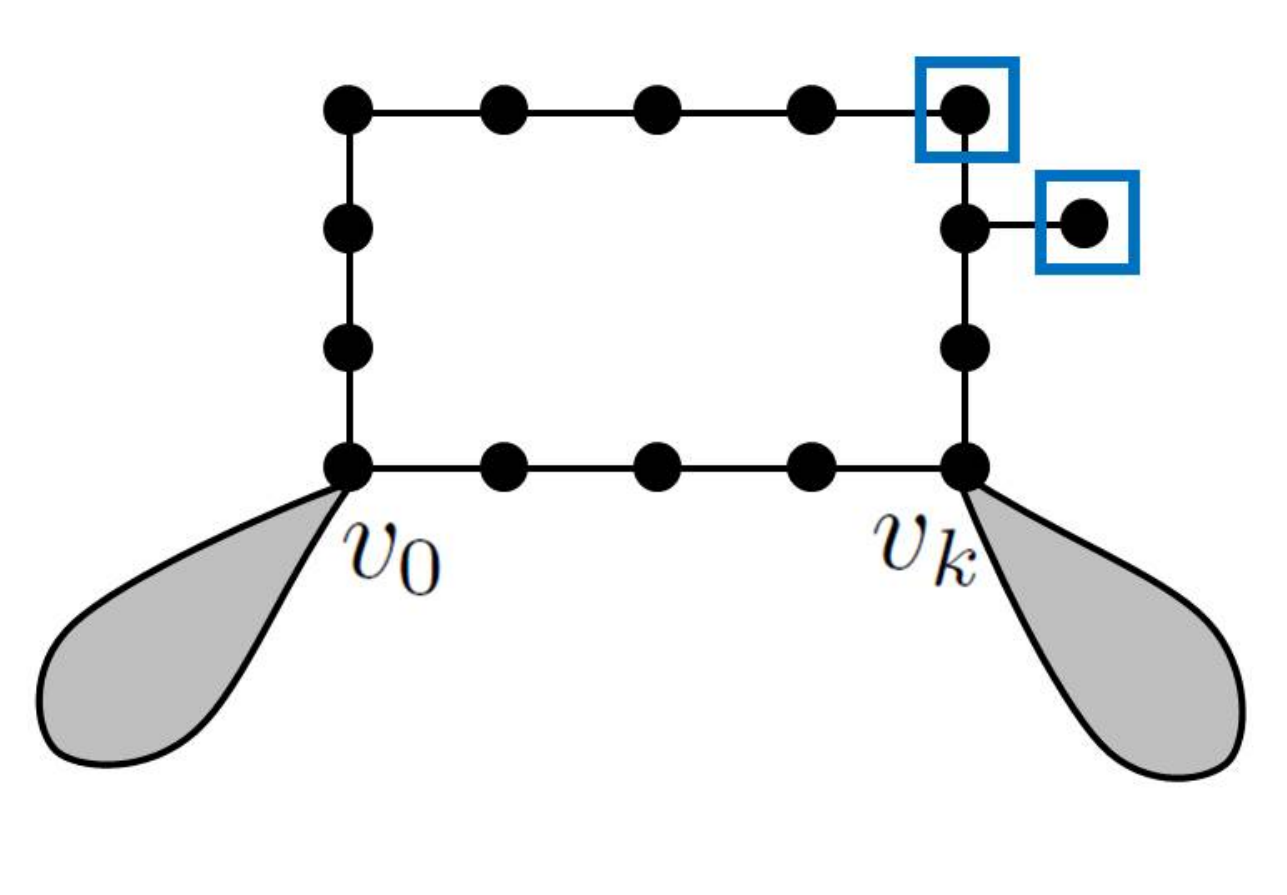}}}\\
\text{c)} &
\text{\raisebox{-1\height}{\includegraphics[scale=0.5]{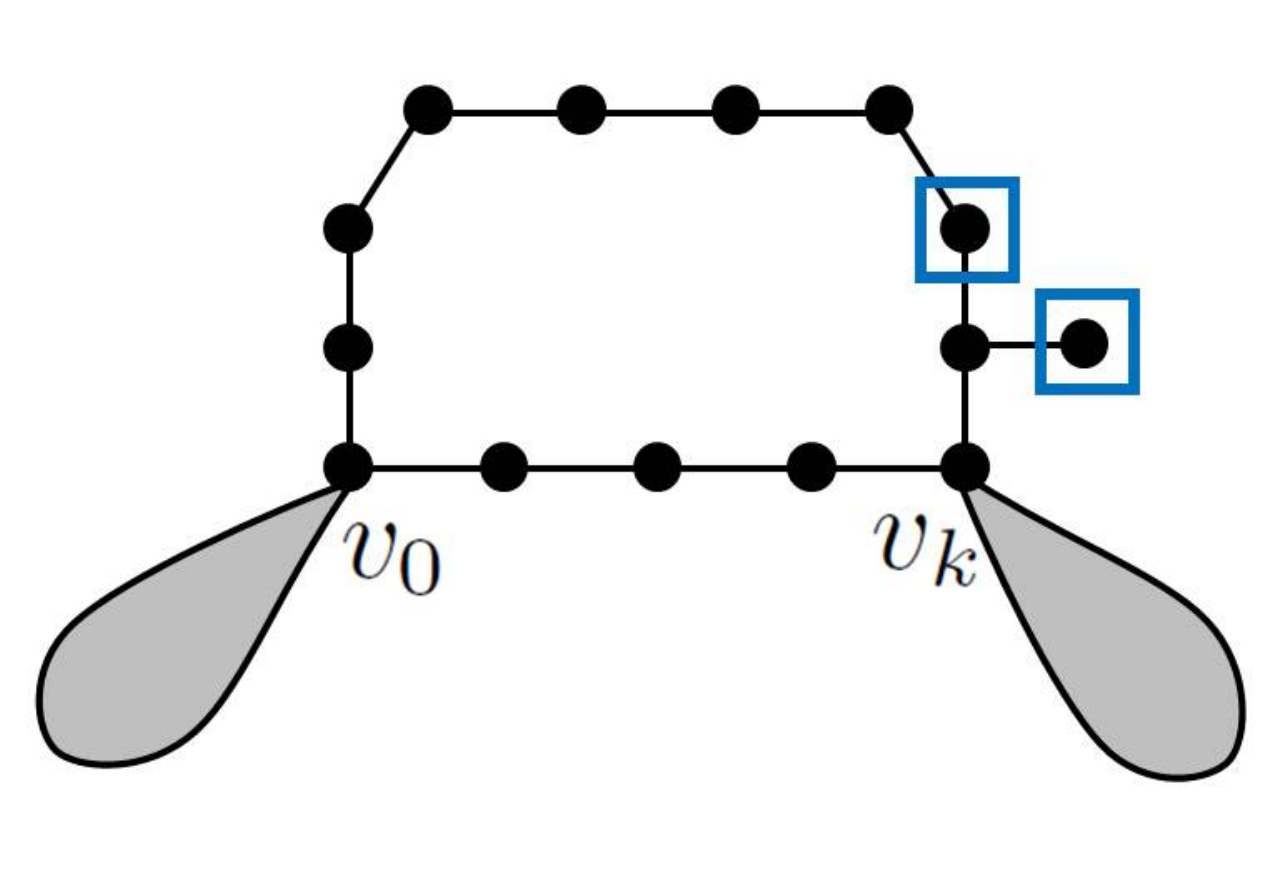}}} &
\text{d)} &
\text{\raisebox{-1\height}{\includegraphics[scale=0.5]{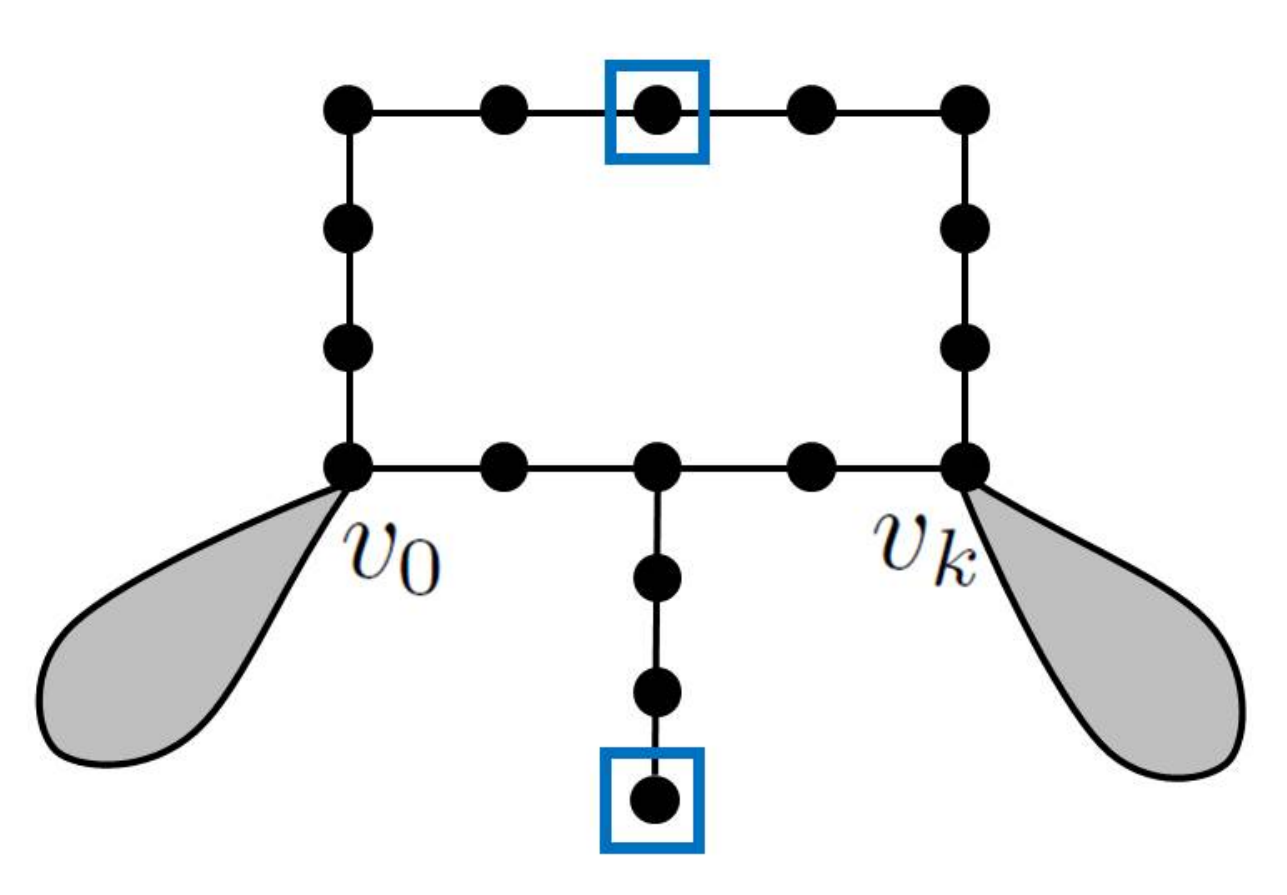}}}%
\end{array}
$
\end{center}
\caption{In all examples we consider a branch-resolving set $S,$ where $v_{0}$
and $v_{k}$ are the only two $S$-active vertices on $C.$ Configuration
$\mathcal{A}$ with respect to $S$ is illustrated by a). Configuration
$\mathcal{B}$ is shown: b) for even cycle, c) for odd cycle. Finally,
configuration $\mathcal{C}$ is illustrated by d). In every graph a pair of
vertices is marked which is not distinguished by $S.$}%
\label{Figure_YesProblem}%
\end{figure}

The above configurations are illustrated by Figure \ref{Figure_YesProblem}. In
every graph from Figure \ref{Figure_YesProblem}, a pair of vertices is marked
which is not distinguished by $S.$ In the next theorem we will show that this
holds in general, i.e. that every set $S$ for which the cycle $C$ has one of
these configurations is not a vertex metric generator, and otherwise $S$ is a
vertex metric generator.

\begin{lemma}
\label{Lm_vertexGeneratorNec}Let $G$ be a unicyclic graph and let $S$ be a
biactive branch-resolving set in $G$. If $G$ contains configuration
$\mathcal{A}$, $\mathcal{B}$, or $\mathcal{C}$ with respect to $S,$ then $S$
is not a vertex metric generator in $G.$
\end{lemma}

\begin{proof}
Let us assume that $G$ contains configuration $\mathcal{A}$, $\mathcal{B}$, or
$\mathcal{C}$ with respect to $S$ and it is sufficient to find a pair of
vertices $x,x^{\prime}\in V(G)$ which are not distinguished by $S$.

If $G$ contains configuration $\mathcal{A}$ then $x=v_{1}$ and $x^{\prime
}=v_{g-1}$ are not distinguished by $S$.

Next, if $G$ contains configuration $\mathcal{B}$, let $v_{i}$ be a vertex
from $C$ with an $S$-free thread hanging at it, where $i\in\lbrack
k,\left\lfloor g/2\right\rfloor -1]\cup\lbrack\left\lceil g/2\right\rceil
+k+1,g]\cup\{0\}$. Let $w$ be the neighbour of $v_{i}$ which belongs to the
thread hanging at $v_{i}$. If $i\in\lbrack k,\left\lfloor g/2\right\rfloor
-1]$ then $x=w$ and $x^{\prime}=v_{i+1}$ are not distinguished by $S$. And if
$i\in\lbrack\left\lceil g/2\right\rceil +k+1,g-1]\cup\{0\},$ then $x=w$ and
$x^{\prime}=v_{i-1}$ are not distinguished by $S$.

Finally, suppose that $G$ contains configuration $\mathcal{C}$. Let $v_{i}$
with $i\in\lbrack0,k]$ be a vertex on the cycle $C$ with an $S$-free thread of
the length $\geq g/2-k$ hanging at it. Let $x$ be the vertex on that thread
such that $d(x,v_{i})=g/2-k$, let $j=2k-i+d(x,v_{i})=g/2+k-i$ and let
$x^{\prime}=v_{j}$. We argue that $x$ and $x^{\prime}$ are not distinguished
by $S$. Note that $i\in\lbrack0,k]$ implies $j\in\left[  g/2,g/2+k\right]  $,
therefore $d(v_{j},v_{k})=j-k$ and $d(v_{j},v_{0})=g-j$. Now a simple
calculation yields that $x$ and $x^{\prime}$ are not distinguished by
$\{v_{0},v_{k}\},$ and therefore they are not distinguished by $S$.
\end{proof}

In the above lemma we have shown that configurations $\mathcal{A}$,
$\mathcal{B}$, and $\mathcal{C}$ are the obstacles for $S$ to be a vertex
metric generator in $G.$ Let us now prove that these configurations are the
only such.

\begin{lemma}
\label{Lm_vertexGeneratorSuf}Let $G$ be a unicyclic graph and let $S$ be a
biactive branch-resolving set in $G$. If $G$ does not contain any of the
configurations $\mathcal{A}$, $\mathcal{B}$, and $\mathcal{C}$ with respect to
$S,$ then the set $S$ is a vertex metric generator in $G$.
\end{lemma}

\begin{proof}
Let us assume that $G$ does not contain any of the configurations
$\mathcal{A}$, $\mathcal{B}$, and $\mathcal{C}$ with respect to $S$ and let us
suppose the contrary to the claim, i.e. $S$ is not a vertex metric generator.
By Lemma \ref{Lemma_GeodTrip}, there is no geodesic triple of $S$-active
vertices on $C$, and hence the canonical labelling of $G$ implies $k\leq g/2$.

Suppose first that $k=g/2$, which implies that $g$ is even. If $a_{S}(C)\geq
3$, then $k=g/2$ implies that the third $S$-active vertex on $C$ together with
$v_{0}$ and $v_{k}$ forms a geodesic triple of $S$-active vertices on $C,$
which is a contradiction. On the other hand, if $a_{S}(C)=2$, then $k=g/2$
implies that the graph $G$ contains the configuration $\mathcal{A}$ which is
again a contradiction.

Suppose now that $k<g/2$. As we assumed that $S$ is not a vertex metric
generator, there must exist a pair of vertices $x$ and $x^{\prime}$ in $G$
which is not distinguished by $S$. Let us assume that $x\in V(T_{v_{i}})$ and
$x^{\prime}\in V(T_{v_{j}})$. If $i=j$, then $x$ and $x^{\prime}$ would be
distinguished by $S$ according to Lemma \ref{Lemma_SameComponent}. Therefore,
without loss of generality, we may assume $i<j$. Now, we consider the
following five cases regarding $i$ and $j$ in order to conclude the proof.

\medskip\noindent\textbf{Case 1:} $i,j\in\left[  0,k\right]  $. If $x$ and
$x^{\prime}$ are not distinguished by $S\cap V(T_{v_{0}})$, then $i<j$ implies
$d(x,v_{i})>d(x^{\prime},v_{j})$ which further implies $d(x,s)>d(x^{\prime
},s)$ for every $s\in S\cap V(T_{v_{k}})$. Therefore, $S$ distinguishes $x$
and $x^{\prime}$ which is a contradiction.

\medskip\noindent\textbf{Case 2:} $i,j\in\left[  k+1,g-1\right]  $. Since we
do not have a configuration $\mathcal{B}$, it must be $\left\lfloor
g/2\right\rfloor \leq i<j\leq\left\lceil g/2\right\rceil +k$. But then notice
the following. If $d(v_{j},x^{\prime})<d(v_{i},x)$ then $v_{0}$ distinguish
these two vertices as $x^{\prime}$ is closer to $v_{0}$ than $x$. And,
similarly if $d(v_{j},x^{\prime})>d(v_{i},x)$ then $v_{k}$ distinguishes these
two vertices as $x$ is closer. So we infer $d(v_{j},x^{\prime})=d(v_{i},x)$.
Then $v_{0}$ does not distinguish $x$ and $x^{\prime}$ only if $g$ is odd and
$v_{i}$ and $v_{j}$ are the antipodals of $v_{0}$. But then $v_{i}$ and
$v_{j}$ cannot be antipodals of $v_{k}$, and so $v_{k}$ distinguishes them.

\medskip\noindent\textbf{Case 3:} $i\in\left[  1,k-1\right]  $ \textit{and}
$j\in\left[  k+1,g-1\right]  $. If $j\leq\left\lfloor g/2\right\rfloor ,$ then
the fact that $S\cap V(T_{v_{k}})$ does not distinguish $x$ and $x^{\prime}$
implies $d(x,v_{i})<d(x^{\prime},v_{i}),$ so $x$ and $x^{\prime}$ are
distinguished by $S\cap V(T_{v_{0}}).$ The similar argument holds if
$j\geq\left\lceil g/2\right\rceil +k.$ Therefore, we may assume that
$j\in\lbrack\left\lfloor g/2\right\rfloor +1,\left\lceil g/2\right\rceil
+k-1],$ which implies that $d(v_{j},v_{0})+d(v_{j},v_{k})=g-k.$ Now, the facts
that $S\cap V(T_{v_{0}})$ and $S\cap V(T_{v_{k}})$ do not distinguish $x$ and
$x^{\prime}$ imply
\begin{align*}
d(x,v_{i})+d(v_{i},v_{0})  &  =d(x^{\prime},v_{j})+d(v_{j},v_{0})\\
d(x,v_{i})+d(v_{i},v_{k})  &  =d(x^{\prime},v_{j})+d(v_{j},v_{k}),
\end{align*}
respectively. Summing these two equalities further implies $d(x,v_{i}%
)-d(x^{\prime},v_{j})=g/2-k$. The fact $k<g/2$ implies $g/2-k>0$, so plugging
this expression in the above equalities we obtain $d(v_{j},v_{0}%
)>d(v_{i},v_{0})$ and $d(v_{j},v_{k})>d(v_{i},v_{k})$. Now, in the case when
$a_{S}(C)\geq3$, there is $l\in\left(  0,k\right)  $ such that $v_{l}$ is
$S$-active, but then for $l\in\left(  0,i\right]  $ the fact that $S\cap
V(T_{v_{0}})$ does not distinguish $x$ and $x^{\prime}$ implies $S\cap
V(T_{v_{l}})$ distinguishes them which is a contradiction, and for
$l\in\left[  i,k\right)  $ when $S\cap V(T_{v_{k}})$ instead of $S\cap
V(T_{v_{0}})$ a similar argument holds. So, we may assume that $a_{S}(C)=2$.
But in this case the fact $d(x,v_{i})-d(x^{\prime},v_{j})=g/2-k>0$ implies
$g/2-k$ is an integer so $g$ must be even and also it implies that there is a
sufficiently long $S$-free thread hanging at $v_{i}$ for $G$ to contain
configuration $\mathcal{C}$ which is a contradiction.

\medskip\noindent\textbf{Case 4:} $i=k$ \textit{and} $j>k$. The fact that
$S\cap V(T_{v_{k}})$ does not distinguish $x$ and $x^{\prime}$ implies
$d(x,v_{k})>d(x^{\prime},v_{j})$, which further implies $x\not =v_{k}$. Notice
that $x$ does not belong to an $S$-free thread hanging at $v_{k}$ as that
would mean $G$ contains configuration $\mathcal{B}$ in all cases except when
$g$ is odd and $k=\left\lfloor g/2\right\rfloor ,$ but in that case
$d(x,v_{k})>d(x^{\prime},v_{j})$ implies $d(x,v_{0})>d(x^{\prime},v_{0}),$ so
$x$ and $x^{\prime}$ are distinguished by $S\cap V(T_{v_{0}}),$ which is a
contradiction. Since $x$ does not belong to an $S$-free thread, we conclude
there must exist a vertex $s\in S\cap V(T_{v_{k}})$ distinct from $v_{k}$ such
that the shortest path $P$ from $x$ to $s$ does not contain $v_{k}.$

Let $v$ be the vertex on the path $P$ which is closest to $v_{k}$. Since $x$
and $x^{\prime}$ are not distinguished by $S\cap V(T_{v_{k}})$, by definition
we have $d(x,s)=d(x^{\prime},s),$ which implies%
\begin{equation}
d(x,v)=d(x^{\prime},v_{k})+d(v_{k},v) \label{Eq1}%
\end{equation}
and hence
\begin{equation}
d(x,v_{k})=d(x^{\prime},v_{k})+2d(v_{k},v). \label{Eq2}%
\end{equation}
The equality (\ref{Eq2}) implies $d(x,v_{k})>d(x^{\prime},v_{k})$. If
$j\leq\left\lfloor g/2\right\rfloor $, then the shortest path from both $x$
and $x^{\prime}$ to $v_{0}$ leads through $v_{k}$, so $d(x,v_{k})>d(x^{\prime
},v_{k})$ would imply that $x$ and $x^{\prime}$ are distinguished by $S\cap
V(T_{v_{0}})$, a contradiction.

Suppose therefore that $j>\left\lfloor g/2\right\rfloor $. In this case a
shortest path from $x^{\prime}$ to $v_{0}$ does not lead through $v_{k},$ so
we have $d(x^{\prime},v_{0})=d(x^{\prime},v_{j})+g-j$. Also, equality
(\ref{Eq1}) implies
\begin{align*}
d(x,v_{0})  &  =d(x,v)+d(v,v_{k})+k\\
&  =d(x^{\prime},v_{k})+2d(v,v_{k})+k\\
&  =d(x^{\prime},v_{j})+j+2d(v,v_{k}).
\end{align*}
Subtracting these expressions for $d(x,v_{0})$ and $d(x^{\prime},v_{0})$ we
obtain%
\[
d(x^{\prime},v_{0})-d(x,v_{0})=g-2j-2d(v,v_{k}),
\]
where the fact that $d(v_{k},v)\not =0$ further implies $d(x^{\prime}%
,v_{0})-d(x,v_{0})\leq g-2j-2$. Note that $j>\left\lfloor g/2\right\rfloor $
implies $g-2j-2<0$, so we conclude that $S\cap V(T_{v_{0}})$ distinguishes $x$
and $x^{\prime}$ which is a contradiction.

\medskip\noindent\textbf{Case 5:} $i=0$ \textit{and} $j>k$. This case is
analogous to Case 4.

\bigskip By the above analysis, we have shown that any pair of vertices from
$G$ is distinguished by $S,$ so $S$ is a vertex metric generator, which
concludes the proof.
\end{proof}

The last two lemmas give us the necessary and sufficient condition for a
biactive branch-resolving set of vertices $S$ to be a vertex metric generator.
In order to establish the exact value of the vertex metric dimension for a
unicyclic graph $G$ we have to find a smallest set $S$ which meets the
condition from Lemmas \ref{Lm_vertexGeneratorNec} and
\ref{Lm_vertexGeneratorSuf}.

Notice that a branch-resolving set $S$ activates all branch-active vertices in
$G,$ so if $b(G)\geq2$ then every branch-resolving set is biactive. Therefore,
for a smallest branch-resolving set $S$ in a unicyclic graph with $b(G)\geq2$,
the set of $S$-active vertices is fixed by the structure of $G$ and coincides
with the set of branch-active vertices. Since the presence of configurations
$\mathcal{A}$, $\mathcal{B}$, and $\mathcal{C}$ by definition depends on the
position of $S$-active vertices, we can observe the following.

\begin{remark}
\label{Obs_b2_notdependS}If a unicyclic graph $G$ contains at least two
branch-active vertices on $C$, i.e. $b(G)\geq2$, then the graph $G$ contains
configuration $\mathcal{A}$, $\mathcal{B}$, or $\mathcal{C}$ either with
respect to all smallest biactive branch-resolving sets or to none of them.
\end{remark}

The above observation allows us to omit the set $S$ in the definition of
containment of configurations $\mathcal{A}$, $\mathcal{B}$, and $\mathcal{C}$
in a unicyclic graph $G$. We can simply say "$G$ contains a configuration"
instead of "$G$ contains a configuration with respect to $S$". Now, for
unicyclic graphs with at least two branch-active vertices we can state and
prove the following theorem which gives the exact value of the vertex metric dimension.

\begin{theorem}
\label{Tm_vDim_bc2}Let $G$ be a unicyclic graph with at least two
branch-active vertices, i.e. $b(G)\geq2$. Then%
\[
\mathrm{dim}(G)=L(G)+\Delta,
\]
where $\Delta=0$ if the graph $G$ does not contain any of the configurations
$\mathcal{A}$, $\mathcal{B}$, $\mathcal{C}$, and $\Delta=1$ otherwise.
\end{theorem}

\begin{proof}
Let $S$ be a smallest branch-resolving set in $G.$ Since $G$ has at least two
branch-active vertices on $C$, i.e. $b(G)\geq2,$ it follows that $S$ is
biactive, and so $a_{S}(C)=b(G).$ If the graph $G$ does not contain any of the
configurations $\mathcal{A}$, $\mathcal{B}$, and $\mathcal{C}$, then Lemma
\ref{Lm_vertexGeneratorSuf} implies that $S$ is a vertex metric generator.
Therefore, $\mathrm{dim}(G)=\left\vert S\right\vert =L(G).$

On the other hand, if $G$ contains any of the configurations $\mathcal{A}$,
$\mathcal{B}$, or $\mathcal{C}$, then $S$ is not a vertex metric generator
according to Lemma \ref{Lm_vertexGeneratorNec}. Let $v$ be a vertex from $C$
which forms a geodesic triple with two branch-active vertices on $C,$ and let
$S^{\prime}=S\cup\{v\}.$ Notice that according to Lemma \ref{Lemma_GeodTrip},
the set $S^{\prime}$ is a vertex metric generator, therefore $\mathrm{dim}%
(G)=\left\vert S^{\prime}\right\vert =L(G)+1.$
\end{proof}

\begin{figure}[h]
\begin{center}
$%
\begin{array}
[c]{cccc}%
\text{a)} &
\text{\raisebox{-1\height}{\includegraphics[scale=0.5]{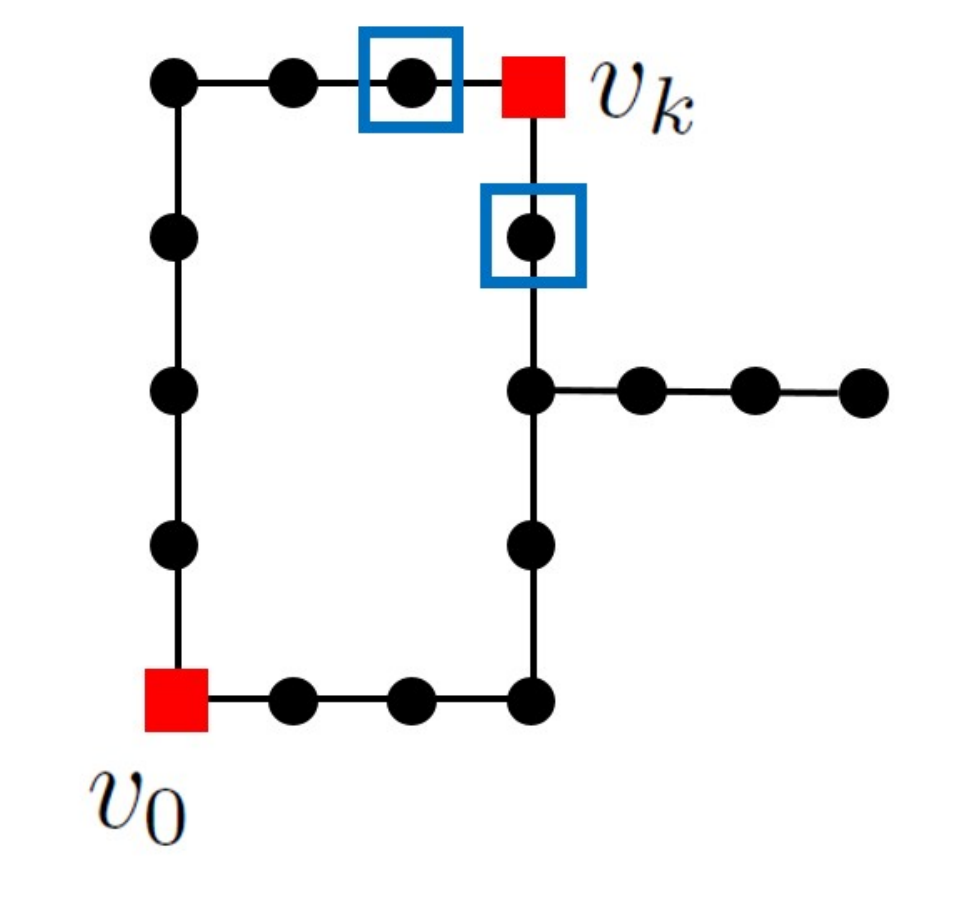}}} &
\text{b)} &
\text{\raisebox{-1\height}{\includegraphics[scale=0.5]{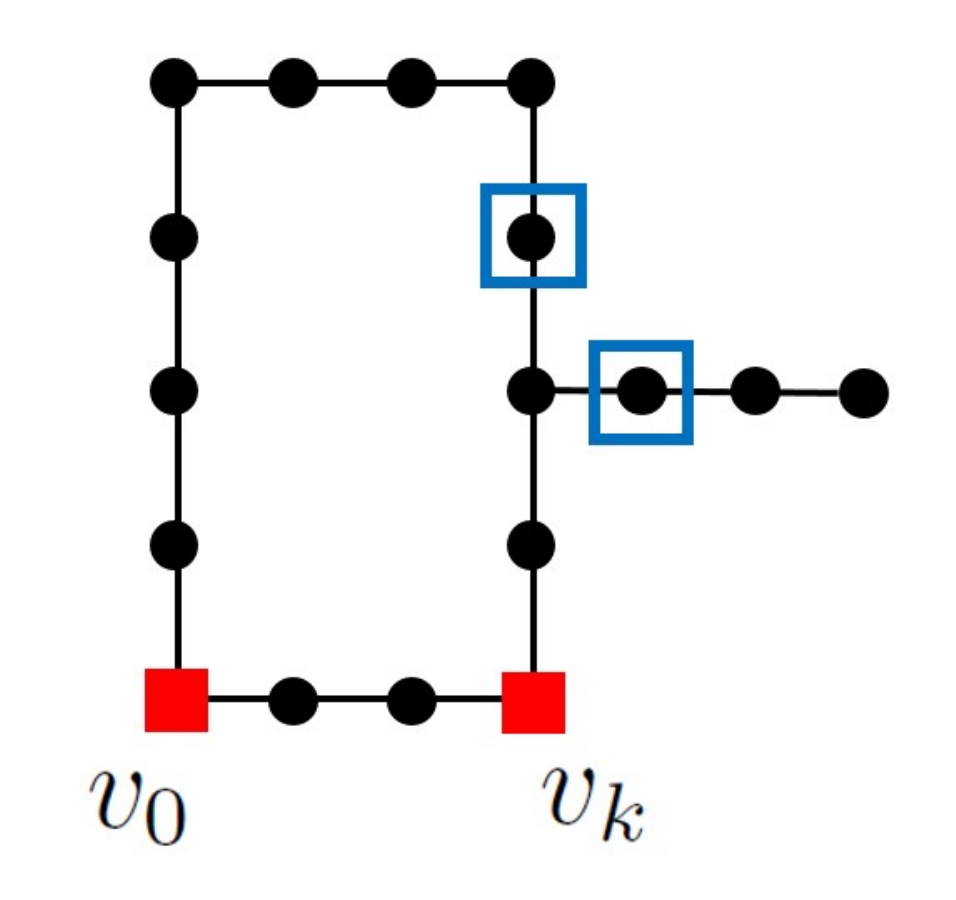}}}\\
\text{c)} &
\text{\raisebox{-1\height}{\includegraphics[scale=0.5]{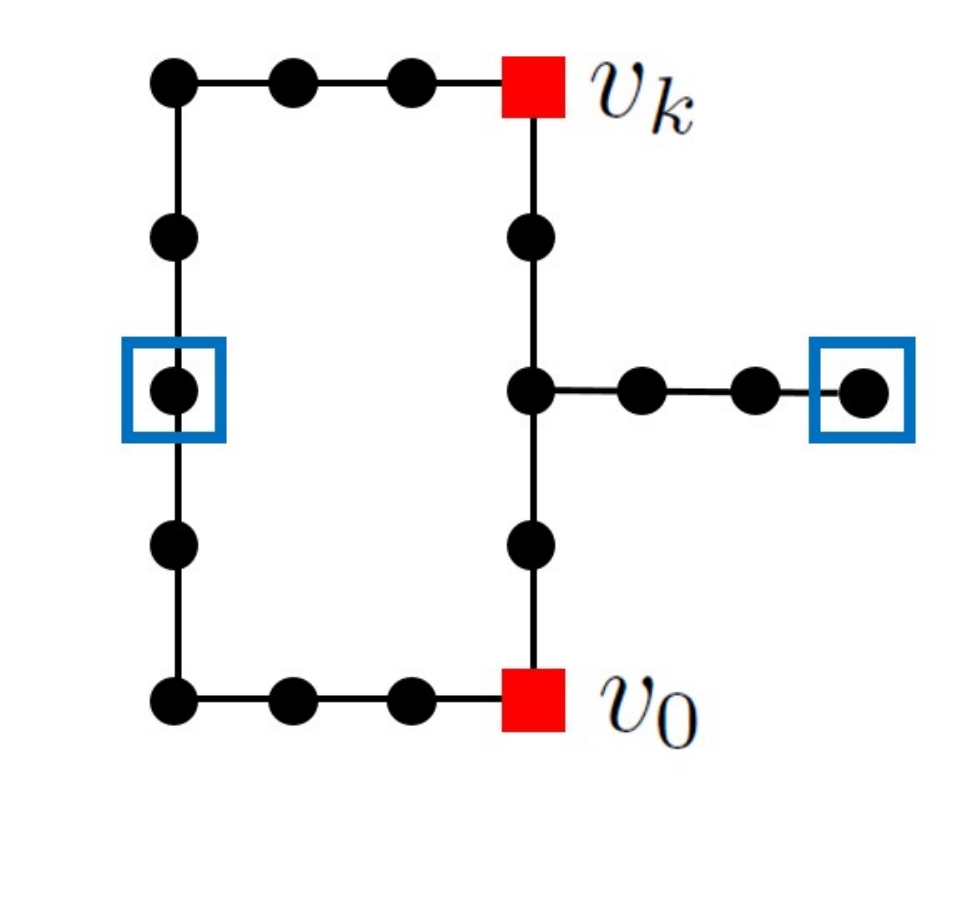}}} &
\text{d)} &
\text{\raisebox{-1\height}{\includegraphics[scale=0.5]{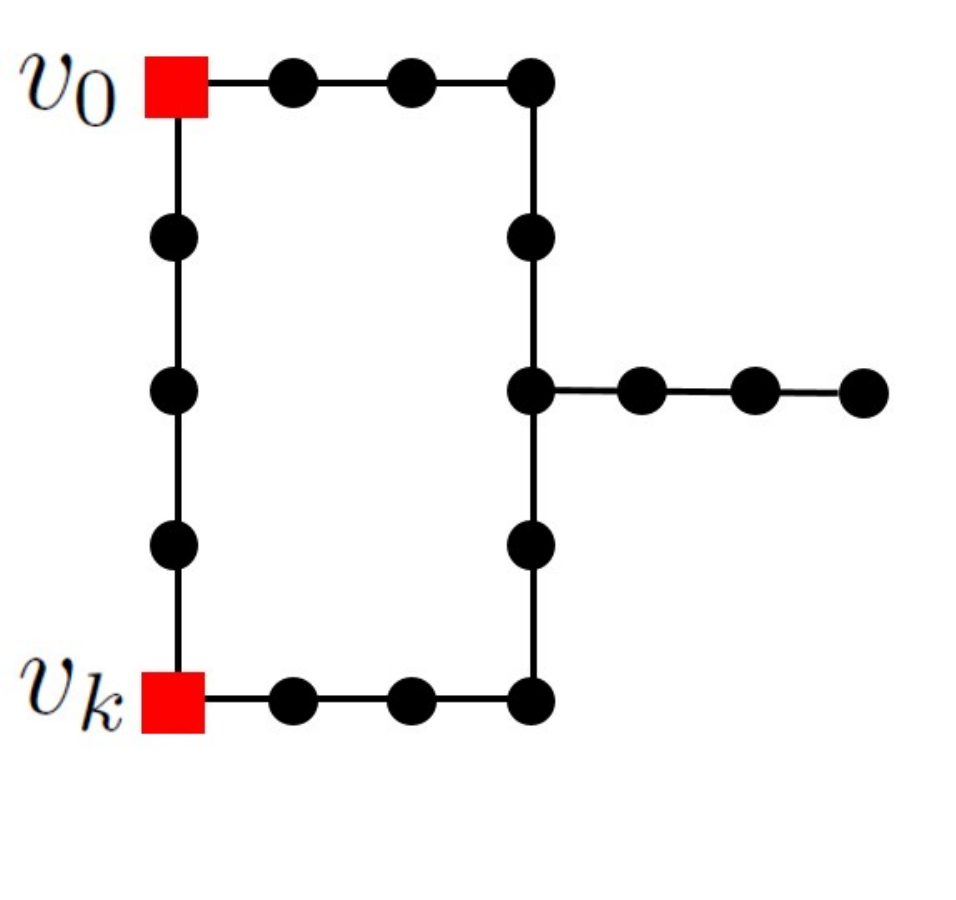}}}%
\end{array}
$
\end{center}
\caption{In all four examples we consider the same graph~$G$ without
branch-active vertices on $C$ and four different smallest biactive
branch-resolving sets $S=\{v_{0},v_{k}\}$ chosen so that the graph $G$ with
respect to $S$ contains: a) configuration $\mathcal{A}$, b) configuration
$\mathcal{B}$, c) configuration $\mathcal{C}$, d) none of the configurations
$\mathcal{A}$, $\mathcal{B}$, and $\mathcal{C}$. In the first three examples
the set $S$ is not a vertex metric generator, so a pair of vertices
non-distinguished by $S$ is marked in $G.$}%
\label{Fig_b01}%
\end{figure}

So far we have determined the exact value of unicyclic graphs $G$ with
$b(G)\geq2$ and our characterization depends on the presence of configurations
$\mathcal{A}$, $\mathcal{B}$, and $\mathcal{C}$ in the graph $G.$ Now we want
to deal with unicyclic graphs $G$ with $b(G)<2.$ The problem with such graphs
is that different biactive sets $S$ impose different $S$-active vertices on
$C,$ and we derive presence of different configurations (see Figure
\ref{Fig_b01}).

We conclude that we cannot speak of presence/absence of a particular
configuration in a graph, unless we have a biactive branch-resolving set $S$.
So, we introduce the following definition to unify both cases $b(G)\geq2$ and
$b(G)<2.$

\begin{definition}
Let $G$ be a unicyclic graph. We say that $G$ is $\mathcal{ABC}$%
\emph{-negative}, if there exists a smallest biactive branch-resolving set $S$
such that $G$ does not contain any of the configurations $\mathcal{A},$
$\mathcal{B},$ and $\mathcal{C}$ with respect to $S.$ Otherwise we say that
$G$ is $\mathcal{ABC}$\emph{-positive}.
\end{definition}

Notice that the above definition extends the case when $G$ contains two or
more branch-active vertices, i.e. $b(G)\geq2$, in which case $G$ will be
$\mathcal{ABC}$-negative if it does not contain any of the configurations
$\mathcal{A},$ $\mathcal{B},$ $\mathcal{C}$, and $G$ is $\mathcal{ABC}%
$-positive if it contains at least one of the configurations $\mathcal{A},$
$\mathcal{B},$ $\mathcal{C}.$

As an example of $\mathcal{ABC}$-negative and $\mathcal{ABC}$-positive graphs
$G$ with $b(G)<2$, we can consider corona product graphs $C_{n}\odot K_{1}$
which are obtained from the cycle $C_{n}$ by appending a leaf to every vertex
in $C_{n}.$ We leave to the reader to verify that if $n$ is odd, then corona
product $C_{n}\odot K_{1}$ is $\mathcal{ABC}$-negative, and for even $n\geq6$
it is $\mathcal{ABC}$-positive (see Figure \ref{Fig_corona}). Corona product
graphs $C_{n}\odot K_{1}$ are examples of unicyclic graphs with $b(G)=0,$ but
if we replace one leaf in $C_{n}\odot K_{1}$ by any acyclic structure, we
obtain a graph with $b(G)=1$ and the same reasoning holds.

\begin{figure}[h]
\begin{center}
$%
\begin{array}
[c]{cccc}%
\text{a)} &
\text{\raisebox{-1\height}{\includegraphics[scale=0.5]{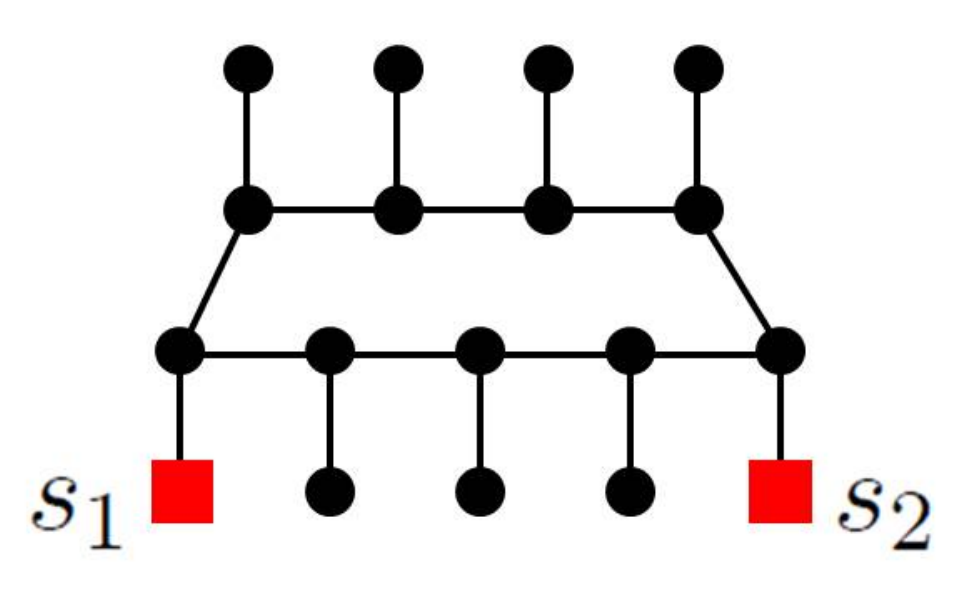}}} &
\text{b)} &
\text{\raisebox{-1\height}{\includegraphics[scale=0.5]{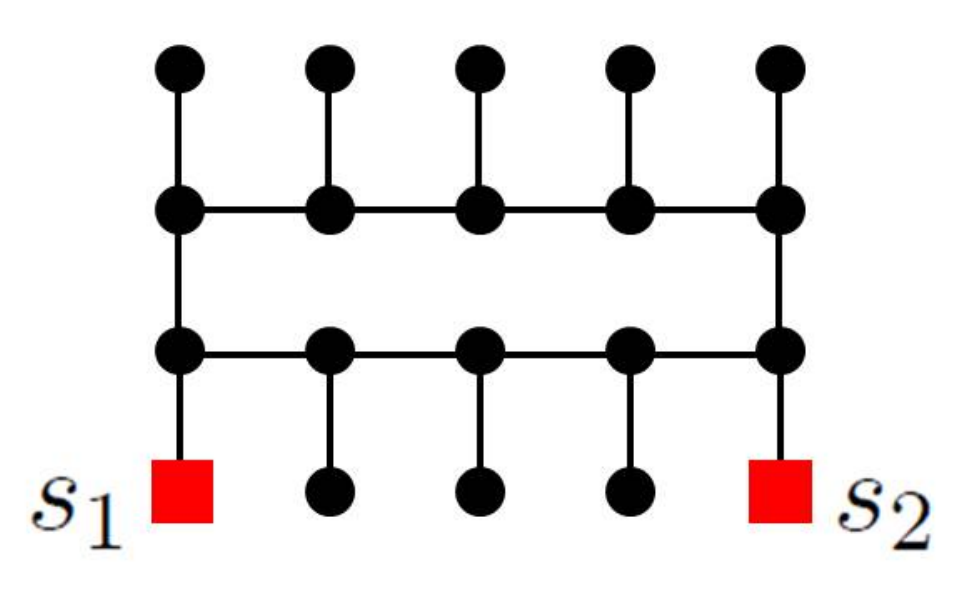}}}%
\end{array}
$
\end{center}
\caption{Figure shows corona product graphs: a) $C_{9}\odot K_{1}$ which is
$\mathcal{ABC}$-negative and $S=\{s_{1},s_{2}\}$ is a vertex metric generator,
b) $C_{10}\odot K_{1}$ which is $\mathcal{ABC}$-positive with $S=\{s_{1}%
,s_{2}\}$ being such that it avoids configurations $\mathcal{A}$ and
$\mathcal{B}$, but does not avoid configuration $\mathcal{C}.$}%
\label{Fig_corona}%
\end{figure}

Now, we can state a more general version of Theorem \ref{Tm_vDim_bc2} which
encapsulates unicyclic graphs with less than two branch-active vertices.

\begin{theorem}
\label{Tm_dim}Let $G$ be a unicyclic graph. Then%
\[
\mathrm{dim}(G)=L(G)+\max\{0,2-b(G)\}+\Delta,
\]
where $\Delta=0$ if the graph $G$ is $\mathcal{ABC}$-negative, and $\Delta=1$
if $G$ is $\mathcal{ABC}$-positive.
\end{theorem}

\begin{proof}
Assume first that $G$ is $\mathcal{ABC}$-negative. This implies that there is
a smallest biactive branch-resolving set $S$ such that $G$ does not contain
any of the configurations $\mathcal{A},$ $\mathcal{B},$ and $\mathcal{C}$ with
respect to $S.$ Then, Lemma \ref{Lm_vertexGeneratorSuf} implies that $S$ is a
vertex metric generator, so $\mathrm{dim}(G)=\left\vert S\right\vert
=L(G)+\max\{0,2-b(G)\}.$

Assume now that $G$ is $\mathcal{ABC}$-positive and let $S$ be a smallest
biactive branch-resolving set in $G.$ Definition of $\mathcal{ABC}$-positivity
implies that $G$ contains at least one of the configurations $\mathcal{A}$,
$\mathcal{B}$, or $\mathcal{C}$ with respect to $S,$ so $S$ is not a vertex
metric generator according to Lemma \ref{Lm_edgeGeneratorNec}. But then, let
$v$ be a vertex from $C$ which forms a geodesic triple with two $S$-active
vertices on $C,$ and let $S^{\prime}=S\cup\{v\}.$ Now, Lemma
\ref{Lemma_GeodTrip} implies that $S^{\prime}$ is a metric generator, so
$\mathrm{dim}(G)=\left\vert S^{\prime}\right\vert =L(G)+\max\{0,2-b(G)\}+1.$
\end{proof}

\section{Edge metric dimension}

Now we want to apply a similar study for the edge metric dimension of
unicyclic graphs. Similarly as for the vertex metric dimension, Lemma
\ref{Lemma_a(S)vj2} implies that a set $S\subseteq V(G)$ which is not a
branch-resolving set or for which $a_{S}(C)<2$ cannot be an edge metric
generator, and Lemma \ref{Lemma_GeodTrip} implies that a branch-resolving set
$S$ with a geodesic triple of $S$-active vertices on the cycle $C$ certainly
is an edge metric generator. Therefore, it remains to consider
branch-resolving sets $S$ with $a_{S}(C)\geq2,$ but without a geodesic triple
of $S$-active vertices. Such a set may or may not be an edge metric generator,
and we want to establish necessary and sufficient conditions for it to be an
edge metric generator.

Let us consider the graphs from Figure \ref{Figure_YesProblem}.\ Notice that
in the graph $G$ from a), which contains configuration $\mathcal{A}$, there is
a pair of edges incident to $v_{0}$ which is not distinguished by $S$.
Similarly holds for graphs in b) and c) which contain configuration
$\mathcal{B}.$ Moreover, in the example c) where the cycle is odd there will
be a pair of undistinguished edges even if the $S$-free thread is hanging at
$v_{i}$ for $i=\left\lceil g/2\right\rceil -1$ or $i=\left\lfloor
g/2\right\rfloor +k+1.$ So, in the case of edge metric dimension configuration
$\mathcal{B}$ will have to be extended to a new configuration $\mathcal{D}.$
Finally, in the graph $G$ from d) there is no pair of undistinguished edges,
so configuration $\mathcal{C}$ is not an obstacle for $S$ to be an edge metric
generator, but there is a counterpart configuration for the edge metric
dimension which will be configuration $\mathcal{E}$.

\begin{definition}
Let $G$ be a unicyclic graph and let $S$ be a biactive branch-resolving set
$S$ in $G$. We say that the graph $G$ with respect to $S$ \emph{contains} configuration:

\begin{description}
\item $\mathcal{D}$. If $k\leq\left\lceil g/2\right\rceil -1$ and there is an
$S$-free thread hanging at a vertex $v_{i}$ for some $i\in\lbrack
k,\left\lceil g/2\right\rceil -1]\cup\lbrack\left\lfloor g/2\right\rfloor
+k+1,g-1]\cup\{0\}$;

\item {$\mathcal{E}$}. If $a_{S}(C)=2$ and there is an $S$-free thread of the
length $\geq\left\lfloor g/2\right\rfloor -k+1$ hanging at a vertex $v_{i}$
with $i\in\lbrack0,k].$ Moreover, if $g$ is even, an $S$-free thread must be
hanging at the vertex $v_{j}$ with $j=g/2+k-i$.
\end{description}
\end{definition}

Notice that if $G$ contains configuration $\mathcal{B},$ then it certainly
contains configuration $\mathcal{D}$, but the oposite does not hold. As for
configuration $\mathcal{E}$, notice that for odd $g$ we encounter
configuration $\mathcal{E}$ just with a thread hanging at $v_{i}$ (as $j$ is
not integer anyway). Configuration $\mathcal{E}$ is illustrated by Figure
\ref{Figure_YesProblemEdge}, where a pair of undistinguished edges is marked
and we will show that the same holds generally, i.e. that configuration
$\mathcal{E}$ is an obstacle for set $S$ to be an edge metric generator.

\begin{figure}[h]
\begin{center}
$%
\begin{array}
[c]{cccc}%
\text{a)} &
\text{\raisebox{-1\height}{\includegraphics[scale=0.5]{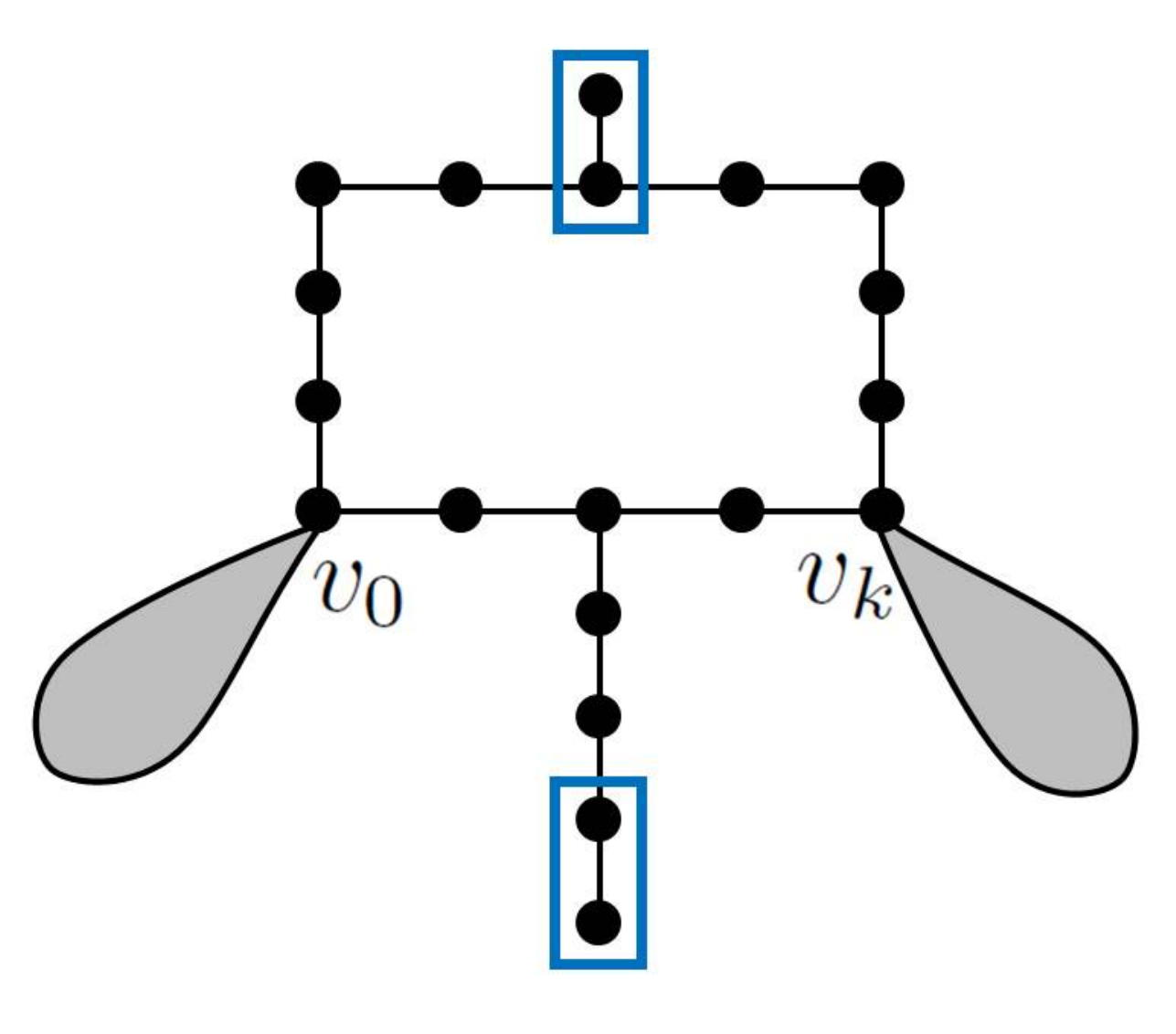}}} &
\text{b)} &
\text{\raisebox{-1\height}{\includegraphics[scale=0.5]{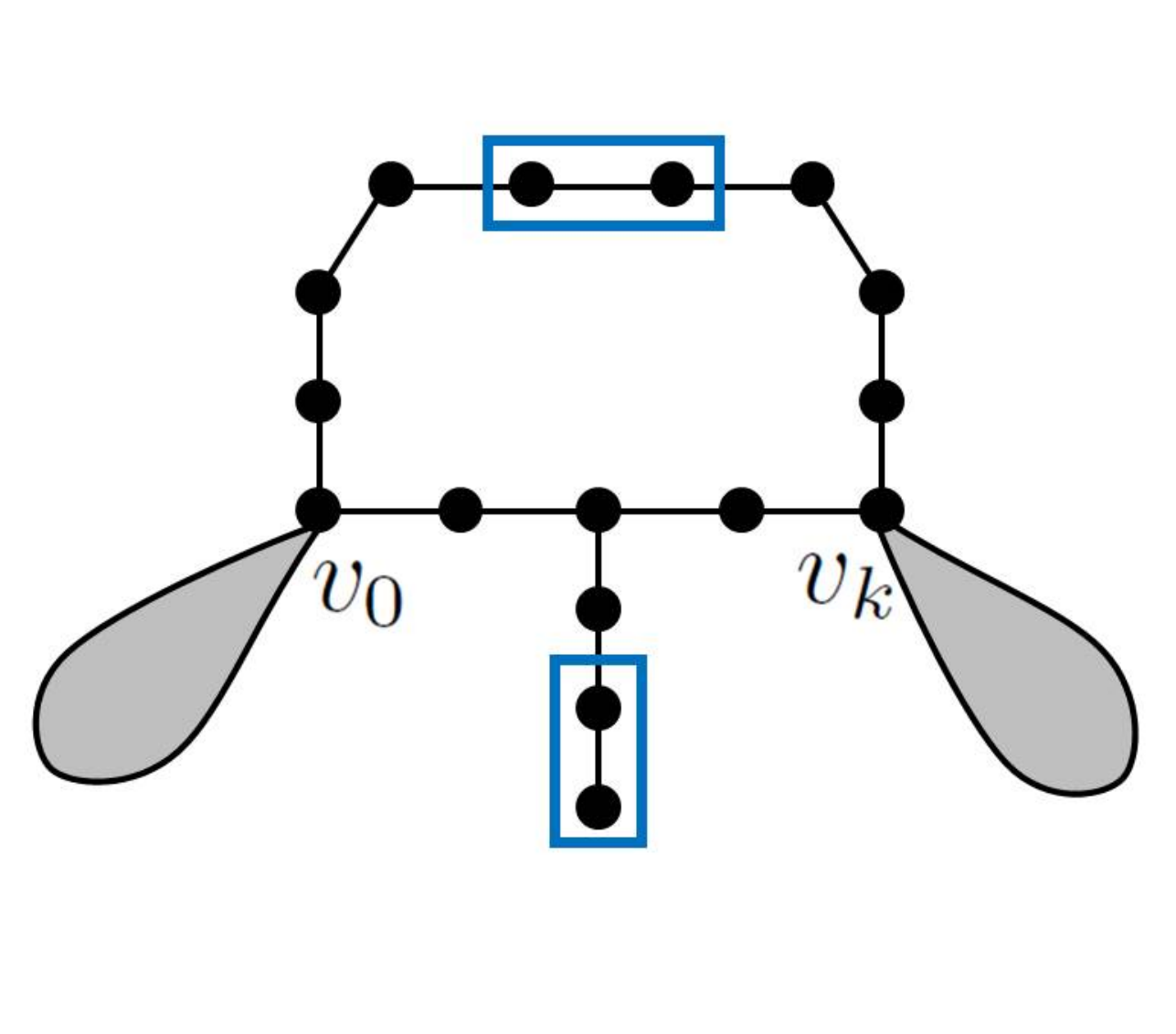}}}%
\end{array}
$
\end{center}
\caption{In both examples we consider a branch-resolving set $S,$ where
$v_{0}$ and $v_{k}$ are the only two $S$-active vertices on $C.$ Configuration
$\mathcal{E}$ with respect to $S$ is shown: a) for an even cycle, b) for an
odd cycle. In both graphs a pair of edges is marked which is not distinguished
by $S.$}%
\label{Figure_YesProblemEdge}%
\end{figure}

\begin{lemma}
\label{Lm_edgeGeneratorNec}Let $G$ be a unicyclic graph and let $S$ be a
biactive branch-resolving set in $G$. If the graph $G$ contains configuration
$\mathcal{A}$, $\mathcal{D}$, or $\mathcal{E}$ with respect to $S,$ then the
set $S$ is not an edge metric generator in $G$.
\end{lemma}

\begin{proof}
Let us assume that $G$ contains configuration $\mathcal{A}$, $\mathcal{D}$, or
$\mathcal{E}$ with respect to $S$, and it is sufficient to find a pair of
edges $x,x^{\prime}\in E(G)$ which are not distinguished by $S$.

If $G$ contains configuration $\mathcal{A}$, then edges $v_{0}v_{1}$ and
$v_{0}v_{g-1}$ are not distinguished by $S$.

Next, if $G$ contains configuration $\mathcal{D}$, let $v_{i}$ be the
"problematic" vertex on $C$ with an $S$-free thread hanging at it and let $w$
be the neighbour of $v_{i}$ on that thread. Then either pair of edges $wv_{i}$
and $v_{i}v_{i+1}$ or the pair $wv_{i}$ and $v_{i}v_{i-1}$ are not
distinguished by $S$.

Finally, assume that $G$ contains configuration $\mathcal{E}$. By definition
this implies that $a_{S}(C)=2$ and there is an $S$-free thread hanging at
$v_{i}$ for $i\in\lbrack0,k]$ of the length $\geq\left\lfloor g/2\right\rfloor
-k+1$ and also an $S$-free thread hanging at $v_{j}$ for $j=g/2+k-i$ (if $g$
is even). Let $e$ be an edge in $T_{v_{i}}$ such that $d(e,v_{i})=\left\lfloor
g/2\right\rfloor -k$, note that such an edge must exist due to the fact that
the thread attached to $v_{i}$ is of length $\geq\left\lfloor g/2\right\rfloor
-k+1$. If $g$ is even, then $v_{j}$ has an $S$-free thread attached, so let
$e^{\prime}$ be the first edge on that thread (i.e. $e^{\prime}$ is incident
to $v_{j}$). Then $e$ and $e^{\prime}$ are not distinguished by $S$. If $g$ is
odd, let $e^{\prime}=v_{\left\lfloor j\right\rfloor }v_{\left\lfloor
j\right\rfloor +1}$. Then again $e$ and $e^{\prime}$ are not distinguished by
$S$.
\end{proof}

So far we have shown that configurations $\mathcal{A}$, $\mathcal{D}$, and
$\mathcal{E}$ are indeed the obstacle for $S$ to be an edge metric generator.
Next, we show that these are the only obstacles.

\begin{lemma}
\label{Lm_edgeGeneratorSuf}Let $G$ be a unicyclic graph, and let $S$ be a
biactive branch-resolving set in $G$. If the graph $G$ does not contain any of
the configurations $\mathcal{A}$, $\mathcal{D}$, and $\mathcal{E}$ with
respect to $S,$ then the set $S$ is an edge metric generator in $G$.
\end{lemma}

\begin{proof}
Suppose that the graph $G$ does not contain any of the configurations
$\mathcal{A}$, $\mathcal{D},$ and $\mathcal{E}$ with respect to $S.$ Let us
suppose the contrary to the claim, i.e. that $S$ is not an edge metric
generator. Let $e=xy$ and $e^{\prime}=x^{\prime}y^{\prime}$ be two edges in
$G$ which are not distinguished by $S$. If $e$ (resp. $e^{\prime}$) is not an
edge of $C,$ then assume that $x$ (resp. $x^{\prime}$) is closer to $C$ than
$y$ (resp. $y^{\prime}$). Also, let $G_{1}=G/\{e,e^{\prime}\}$ and let the
vertex of $G_{1}$ obtained by contracting $e$ (resp. $e^{\prime}$) be denoted
by $x$ (resp. $x^{\prime}$). Denote by $d_{1}(u,v)$ the distance of vertices
$u$ and $v$ in $G_{1}$. The length of the cycle in $G_{1}$ will be denoted by
$g_{1}$. Now we consider the following three cases.

\medskip\noindent\textbf{Case 1:} $e,e^{\prime}\in E(C)$. Edges $e$ and
$e^{\prime}$ are not distinguished by $S$ only if $g$ is even and $a_{S}%
(C)=2$, where the only two $S$-active vertices are an antipodal pair, which
implies $k=g/2$. Hence, we infer that $G$ contains configuration $\mathcal{A}$
which is a contradiction.

\medskip\noindent\textbf{Case 2:} $e,e^{\prime}\not \in E(C)$. Let $e\in
E(T_{v_{i}})$ and $e^{\prime}\in E(T_{v_{j}})$. Lemma
\ref{Lemma_SameComponent} then implies $i\not =j$, where without loss of
generality we may assume $i<j$. Let $G_{y},G_{y^{\prime}}$ and $G_{x}$ be the
connected components of $G-\{e,e^{\prime}\}$ that contains vertices $y$,
$y^{\prime}$ and $x$ respectively. If there is a vertex $s\in S\cap
V(G_{y}\cup G_{y^{\prime}})$, then obviously $s$ distinguishes $e$ and
$e^{\prime}$. So, let us assume $S\subseteq V(G_{x})$ which implies
$d(s,e)=d(s,x)=d_{1}(s,x)$ and $d(s,e^{\prime})=d(s,x^{\prime})=d_{1}%
(s,x^{\prime})$. Hence, $e$ and $e^{\prime}$ are distinguished by $S$ in $G$
if and only if $x$ and $x^{\prime}$ are distinguished by $S$ in $G_{1}$.

As we assumed that $e$ and $e^{\prime}$ are not distinguished by $S$ in $G,$
Lemma \ref{Lm_vertexGeneratorSuf} implies that $G_{1}$ contains configurations
$\mathcal{A},$ $\mathcal{B}$ or $\mathcal{C}$. If $G_{1}$ contains
configuration $\mathcal{A}$ (resp. $\mathcal{B}$), then $G$ obviously contains
configuration $\mathcal{A}$ (resp. $\mathcal{D}$), which is contradiction. On
the other hand, if $x$ and $x^{\prime}$ are not distinguished by $S$ in
$G_{1}$ due to configuration $\mathcal{C},$ then $g=g_{1}$ is even and
$a_{S}(C)=2$. Also, from $i<j$ we infer that $x$ belongs to an $S$-free thread
hanging at a vertex $v_{i}$ for some $i\in\lbrack0,k]$ and $d(x,v_{i})\geq
g/2-k.$ Moreover, since $x$ and $x^{\prime}$ are not distinguished by $S$,
then $x^{\prime}$ must belong to $T_{v_{j}}$ for $j=g/2+k-i.$ Given the fact
that in $G$ there is an edge $e$ and $e^{\prime}$ attached to vertices $x$ and
$x^{\prime}$ respectively, this implies that $G$ contains configuration
$\mathcal{E}$, a contradiction.

\medskip\noindent\textbf{Case 3:} $e\in E(C)$ \textit{and} $e^{\prime
}\not \in E(C)$. Suppose $e=e_{i}=v_{i}v_{i+1}$ and $e^{\prime}\in E(T_{v_{j}%
})$. Let $G_{x^{\prime}}$ and $G_{y^{\prime}}$ be the connected components of
$G-e^{\prime}$ containing vertices $x^{\prime}$ and $y^{\prime}$ respectively.
If there is a vertex $s\in S\cap V(G_{y^{\prime}})$, then $e$ and $e^{\prime}$
would be distinguished by $s$. Therefore assume $S\subseteq V(G_{x^{\prime}}%
)$. If there is a vertex $s\in S$ such that the shortest path connecting
vertices $x^{\prime}$ and $s$ contains $e$, then $e$ and $e^{\prime}$ would
obviously be distinguished by $S$. Therefore, assume that no path from
$x^{\prime}$ to vertices from $S$ contains the edge $e$.

If $e$ and $e^{\prime}$ are incident, say $x^{\prime}=v_{i},$ then $e$ and
$e^{\prime}$ are not distinguished by $S$ only if $i\in\lbrack k,\left\lceil
g/2\right\rceil -1].$ Since $e^{\prime}\not \in E(C),$ this implies that $G$
contains configuration $\mathcal{D}$. So, let us assume that $e$ and
$e^{\prime}$ are not incident which implies $x\not =x^{\prime}$ in $G_{1}.$
Since no path from $x^{\prime}$ to vertices from $S$ contains the edge $e,$ we
conclude that $d(e,s)=d_{1}(x,s)$ and $d(e^{\prime},s)=d_{1}(x^{\prime},s)$
for every $s\in S$. As we assumed that $e$ and $e^{\prime}$ are not
distinguished by $S$ in $G,$ it follows that $x$ and $x^{\prime}$ are not
distinguished by $S$ in $G_{1},$ so according to Lemma
\ref{Lm_vertexGeneratorSuf} the graph $G_{1}$ contains configuration
$\mathcal{A},$ $\mathcal{B},$ or $\mathcal{C}.$ Notice that in this case
$g_{1}=g-1$.

Similarly as in the previous case, if $G_{1}$ contains configuration
$\mathcal{B}$, then $G$ contains configuration $\mathcal{D}$, which is
contradiction. If $G_{1}$ contains configuration $\mathcal{A},$ then $g_{1}$
is even, so $g$ is odd and $k=\left\lfloor g/2\right\rfloor $. Also, as $e\in
E(C)$ and $e^{\prime}\not \in E(C)$, it must hold $i\in\lbrack k,g-1]$ and
$j\in\lbrack0,k].$ Since $x^{\prime}\in T_{v_{j}}$ has an edge $e^{\prime
}\not \in E(C)$ attached to it, this implies there is an $S$-free thread
hanging at $v_{j}$ of the length $\geq1=\left\lfloor g/2\right\rfloor -k+1.$
Therefore, $G$ contains configuration $\mathcal{E}$ on odd cycle, a contradiction.

Finally, let us assume $x$ and $x^{\prime}$ are not distinguished by $S$ in
$G_{1}$ due to configuration $\mathcal{C}.$ This implies $g_{1}$ is even and
$x^{\prime}$ belongs to an $S$-free thread hanging at $v_{j}$ for $j\in
\lbrack0,k]$ and $d(x^{\prime},v_{j})\geq g/2-k.$ Since there is an edge
$e^{\prime}$ hanging at $x^{\prime}$ in $G,$ this implies $G$ contains
configuration $\mathcal{E}$ on odd cycle.

Altogether, we conclude that any two distinct edges are distinguished by $S,$
which implies that $S$ is an edge metric generator.
\end{proof}

In the previous two lemmas we have established the necessary and sufficient
condition for a set of vertices to be an edge metric generator, so now we can
proceed with determining the exact value of the edge metric dimension of a
unicyclic graph $G.$ Notice that configuration $\mathcal{D}$ and $\mathcal{E}%
$, similarly as configurations $\mathcal{A}$, $\mathcal{B}$, and $\mathcal{C}%
$, depend only on the position of $S$-active vertices on $C.$ Therefore,
Observation \ref{Obs_b2_notdependS} holds also for configurations
$\mathcal{D}$ and $\mathcal{E}$ and we can say that unicyclic graphs with
$b(G)\geq2$ contain configuration $\mathcal{D}$ or $\mathcal{E}$ without
explicitely stating the set $S$. So, for unicyclic graphs with $b(G)\geq2$ we
can state the theorem which gives the edge metric dimension as follows.

\begin{theorem}
\label{Tm_eDim_bc2}Let $G$ be a unicyclic graph with at least two
branch-active vertices. Then%
\[
\mathrm{dim}(G)=L(G)+\Delta_{e},
\]
where $\Delta_{e}=0$ if the graph $G$ does not contain any of the
configurations $\mathcal{A}$, $\mathcal{D}$, or $\mathcal{E},$ and $\Delta
_{e}=1$ otherwise.
\end{theorem}

\begin{proof}
The proof is analogous to the proof of Theorem \ref{Tm_vDim_bc2}, using
configurations $\mathcal{D}$ and $\mathcal{E}$ instead of configurations
$\mathcal{B}$ and $\mathcal{C}$ respectively. Also, Lemmas
\ref{Lm_edgeGeneratorNec} and \ref{Lm_edgeGeneratorSuf} need to be applied
instead of Lemmas \ref{Lm_vertexGeneratorNec} and \ref{Lm_vertexGeneratorSuf}, respectively.
\end{proof}

Finally, if a unicyclic graph $G$ contains less than two branch-active
vertices on $C,$ then a smallest branch-resolving set $S$ is not biactive and
needs to be introduced $2-b(G)$ vertices to become biactive, the consequence
of which is that different smallest \emph{biactive} branch-resolving sets $S$
may have different set of $S$-active vertices on $C.$ Since the presence of
configurations $\mathcal{D}$ and $\mathcal{E}$ depends on the position of
$S$-active vertices on $C,$ this implies that a unicyclic graph $G$ with
$b(G)<2$ does contain configuration $\mathcal{D}$ or $\mathcal{E}$ with
respect to one smallest biactive branch-resolving set, but not with respect to
another. Since we need a smallest biactive branch-resolving set such that $G$
does not contain any of the configurations $\mathcal{A}$, $\mathcal{D}$, and
$\mathcal{E}$ with respect to $S,$ we introduce the following definition.

\begin{definition}
We say that a unicyclic graph $G$ is $\mathcal{ADE}$\emph{-negative}, if there
is a smallest biactive branch-resolving set $S$ such that $G$ does not contain
any of the configurations $\mathcal{A}$, $\mathcal{D}$, and $\mathcal{E}$ with
respect to $S.$ Otherwise, we say that $G$ is $\mathcal{ADE}$\emph{-positive}.
\end{definition}

Again, a unicyclic graph $G$ with $b(G)\geq2$ is $\mathcal{ADE}$-negative if
it does not contain any of the configurations $\mathcal{A}$, $\mathcal{D}$,
and $\mathcal{E}$, otherwise it is $\mathcal{ADE}$-positive. In case of
unicyclic graphs with $b(G)<2$, we can again consider corona product graphs
$C_{n}\odot K_{1}$ as an example (see Figure \ref{Fig_corona}). Notice that in
this case it is opposite to the situation with vertex metric dimension, now a
corona graph $C_{n}\odot K_{1}$ with odd $n\geq7$ is $\mathcal{ADE}$-positive
as the set $S=\{s_{1},s_{2}\}$ shown in Figure \ref{Fig_corona}.a) does not
avoid configuration $\mathcal{E}$. On the other hand, a graph $C_{n}\odot
K_{1}$ with even $n$ is $\mathcal{ADE}$-negative (the set $S=\{s_{1},s_{2}\}$
shown in Figure \ref{Fig_corona}.b) avoids all three configurations and is
therefore an edge metric generator).

\begin{theorem}
\label{Tm_edim}Let $G$ be a unicyclic graph. Then%
\[
\mathrm{edim}(G)=L(G)+\max\{0,2-b(G)\}+\Delta_{e}%
\]
where $\Delta_{e}=0$ if the graph $G$ is $\mathcal{ADE}$-negative, and
$\Delta_{e}=1$ if $G$ is $\mathcal{ADE}$-positive.
\end{theorem}

\begin{proof}
Analogous to the proof of Theorem \ref{Tm_dim}.
\end{proof}

\section{Difference between metric dimensions}

Now, we can use the results proven in the previous two sections, and so we
answer a proposal from \cite{SedSkreBounds}. Namely, in \cite{SedSkreBounds}
it was shown that for a unicyclic graph $G$ it holds that $\left\vert
\mathrm{dim}(G)-\mathrm{edim}(G)\right\vert \leq1$ and the problem of
determining whether the difference $\mathrm{dim}(G)-\mathrm{edim}(G)$ is $-1$,
$0$ and $1$ was posed as a natural question. This can be easily answered for
unicyclic graphs $G$ with $b(G)\geq2$ using Theorems \ref{Tm_vDim_bc2} and
\ref{Tm_eDim_bc2}.

\begin{theorem}
\label{Tm_difference_bc2}Let $G$ be a unicyclic graph with $b(G)\geq2$. It
holds that%
\[
\mathrm{dim}(G)-\mathrm{edim}(G)=\left\{
\begin{array}
[c]{rl}%
1 & \text{if }G\text{ contains configuration }\mathcal{C},\text{ but none of
}\mathcal{A}\text{, }\mathcal{D}\text{, and }\mathcal{E}\text{,}\\
-1 & \text{if }G\text{ contains configuration }\mathcal{D}\text{ or
}\mathcal{E},\text{ but none of }\mathcal{A}\text{, }\mathcal{B}\text{, and
}\mathcal{C}\text{,}\\
0 & \text{otherwise.}%
\end{array}
\right.
\]

\end{theorem}

\begin{proof}
According to Theorems \ref{Tm_dim} and \ref{Tm_edim} both $\mathrm{dim}(G)$
and $\mathrm{edim}(G)$ take their value from
\[
L(G)+\max\{2-b(G),0\}\quad\hbox{ and }\quad L(G)+\max\{2-b(G),0\}+1.
\]
Moreover, $\mathrm{dim}(G)$ (resp. $\mathrm{edim}(G)$) will take the greater
value of the two, if $G$ contains configuration $\mathcal{A}$, $\mathcal{B}$,
or $\mathcal{C}$ (resp. configuration $\mathcal{A}$, $\mathcal{D}$, or
$\mathcal{E}$). The observation that $G$ cannot contain configuration
$\mathcal{B}$ without containing configuration $\mathcal{D}$ concludes the proof.
\end{proof}

A more general version of Theorem \ref{Tm_difference_bc2}, which encapsulates
also unicyclic graphs $G$ with $b(G)<2,$ can be established using Theorems
\ref{Tm_dim} and \ref{Tm_edim}.

\begin{theorem}
\label{Tm_difference}Let $G$ be a unicyclic graph. It holds that%
\[
\mathrm{dim}(G)-\mathrm{edim}(G)=\left\{
\begin{array}
[c]{rl}%
1 & \text{if }G\text{ is }\mathcal{ABC}\text{-positive and }\mathcal{ADE}%
\text{-negative,}\\
-1 & \text{if }G\text{ is }\mathcal{ABC}\text{-negative and }\mathcal{ADE}%
\text{-positive,}\\
0 & \text{otherwise.}%
\end{array}
\right.
\]

\end{theorem}

\begin{proof}
It goes similarly as the proof of Theorem \ref{Tm_difference_bc2}.
\end{proof}

An example of graphs $G$ with $b(G)<2$ which are at the same time
$\mathcal{ABC}$-positive and $\mathcal{ADE}$-negative, is the family of corona
graphs $C_{n}\odot K_{1}$ with $n\geq6$ even (see Figure \ref{Fig_corona}).
Therefore, for such graphs we have $\mathrm{dim}(C_{n}\odot K_{1}%
)-\mathrm{edim}(C_{n}\odot K_{1})=3-2=1.$ On the other hand, if $n\geq7$ is
odd, then the corona graphs $C_{n}\odot K_{1}$ can serve as an example of
unicyclic graphs $G$ with $b(G)<2$ which are $\mathcal{ABC}$-negative and
$\mathcal{ADE}$-positive. For them we derive $\mathrm{dim}(C_{n}\odot
K_{1})-\mathrm{edim}(C_{n}\odot K_{1})=2-3=-1$.

It is of interest to determine the classes of graphs for which $\mathrm{dim}%
(G)$ is smaller, equal or bigger from $\mathrm{edim}(G)$. Several such
families of graphs were presented in \cite{TratnikEdge}. We can now make the
same distinction in the class of unicyclic graphs.

\begin{corollary}
\label{Kor_evenOdd}Let $G$ be a unicyclic graph with its cycle $C$. If $C$ is
odd then $\mathrm{dim}(G)\leq\mathrm{edim}(G)$ and the inequality is strict if
$G$ is $\mathcal{ABC}$-negative and $\mathcal{ADE}$-positive. If $C$ is even
then $\mathrm{dim}(G)\geq\mathrm{edim}(G)$ and the inequality is strict if $G$
is $\mathcal{ABC}$-positive and $\mathcal{ADE}$-negative.
\end{corollary}

\begin{proof}
Let $C$ be the cycle in $G$ and assume first that $C$ is of odd length.
According to Theorem \ref{Tm_difference}, the inequality $\mathrm{dim}%
(G)\leq\mathrm{edim}(G)$ will hold if every $G$ which is $\mathcal{ABC}%
$-positive is also $\mathcal{ADE}$-positive. But that is the obvious
consequence of the fact that configuration $\mathcal{B}$ is also configuration
$\mathcal{D},$ i.e. graph $G$ which contains configuration $\mathcal{B}$
certainly contains configuration $\mathcal{D}$ with respect to the same set
$S,$ and the fact that graph on odd cycle cannot contain configuration
$\mathcal{C}$.

Assume now that $C$ is of even length. Again, Theorem \ref{Tm_difference}
implies that the inequality $\mathrm{dim}(G)\geq\mathrm{edim}(G)$ holds if
every $G$ which is $\mathcal{ABC}$-negative is also $\mathcal{ADE}$-negative.
Since configuration $\mathcal{B}$ is also $\mathcal{D}$ and on even cycle
every graph which contains $\mathcal{E}$ also contains $\mathcal{C},$ the
claim follows.

The claim on strictness is the direct consequence of Theorem
\ref{Tm_difference}.
\end{proof}

Regarding Corollary \ref{Kor_evenOdd}, let us mention that in
\cite{KelHypercubes}, it was shown that for bipartite graphs $\mathrm{dim}%
(G)\geq\mathrm{edim}(G).$ Our result extends to all unicylic graphs and
characterizes when the equality holds.

\section{Concluding remarks}

By our previous work, both the vertex and the edge metric dimensions of a
unicyclic graph $G$ takes their value in
\[
L(G)+\max\{2-b(G),0\}\quad\hbox{ and }\quad L(G)+\max\{2-b(G),0\}+1.
\]
In this paper, we first characterize for unicyclic graphs $G$ with $b(G)\geq2$
when the above two values are encountered depending on the presence of
configurations $\mathcal{A}$, $\mathcal{B}$, $\mathcal{C}$, $\mathcal{D}$, and
$\mathcal{E}$. Afterwards, the approach was extended to unicyclic graphs $G$
with $b(G)<2$ by extending the concept of containment of configuration to
$\mathcal{X}$-positivity and $\mathcal{X}$-negativity for $\mathcal{X}%
\in\{\mathcal{ABC},\mathcal{ADE}\}$.

One may try to characterize unicyclic graphs $G$ with $b(G)<2$ and that are
$\mathcal{X}$-positive and $\mathcal{Y}$-negative for distinct $\mathcal{X}%
,\mathcal{Y}\in\{\mathcal{ABC},\mathcal{ADE}\}$, as this may need different
approaches and this paper is already lengthy, we decided not to conduct this
direction of research here. Here we state it explicitly as a problem.

\begin{problem}
Characterize which unicyclic graphs $G$ with $b(G)<2$ are $\mathcal{X}%
$-positive and $\mathcal{Y}$-negative for distinct $\mathcal{X},\mathcal{Y}%
\in\{\mathcal{ABC},\mathcal{ADE}\}$.
\end{problem}

\bigskip

\bigskip\noindent\textbf{Acknowledgements.}~~The authors acknowledge partial
support of the Slovenian research agency ARRS program \ P1--0383 and ARRS
project J1-1692 and also Project KK.01.1.1.02.0027, a project co-financed by
the Croatian Government and the European Union through the European Regional
Development Fund - the Competitiveness and Cohesion Operational Programme.

\end{document}